\documentclass{amsart}
\usepackage{amssymb}

\usepackage{mathabx} % Adds a lot of math symbols and modifies the common ones
\usepackage{microtype} % Sublim­i­nal re­fine­ments to­wards ty­po­graph­i­cal per­fec­tion

\usepackage[unicode,pdftex,colorlinks=true]{hyperref} % Turns \ref and \cite command into hyperlinks
	
\usepackage[T1]{fontenc} % Allows diacritics to be copied from the pdf properly
\usepackage[utf8]{inputenc} % Allows diacritics to be included in the source

\usepackage{enumitem} % Used to customize lists
\usepackage{nicefrac} % Typeset in-line fractions in a "nice" way

\usepackage{mathtools} % for multline

\usepackage{stmaryrd} % Only used for \llbracket and \rrbracket (see \Cyl)

\usepackage{tikz}
\usetikzlibrary{arrows,positioning} 
\tikzset{
	>=stealth',
	imp/.style={
		->,
	}
}

\usepackage{color,soul}
\definecolor{paleblue}{rgb}{0.7804,0.8353,0.9725}
\sethlcolor{paleblue}

\theoremstyle{plain}
\newtheorem{thm}{Theorem}[section]
\newtheorem{prop}[thm]{Proposition}
\newtheorem{cor}[thm]{Corollary}
\newtheorem{lem}[thm]{Lemma}

\theoremstyle{definition}
\newtheorem{defn}[thm]{Definition}

\newtheorem{obs}[thm]{Observation}

\theoremstyle{remark}
\newtheorem{remark}[thm]{Remark}

\numberwithin{equation}{section}

\renewcommand{\epsilon}{\varepsilon}
\renewcommand{\phi}{\varphi}
\renewcommand{\setminus}{\smallsetminus}

\newcommand{\seq}[1]{{\left\langle{#1}\right\rangle}}
\newcommand{\set}[2]{\left\{#1 \,:\, #2\right\}}

\DeclareMathOperator{\dom}{dom}

\DeclareMathOperator{\uh}{\upharpoonright}

\newcommand{\Q}{\mathbb{Q}}
\newcommand{\R}{\mathbb{R}}

\newcommand{\Cyl}[1]{\,\,\,\llbracket{#1}\rrbracket}
\newcommand{\rest}[1]{\! \upharpoonright{#1}} 

\newcommand{\conc}{\hat{\,\,}}
\newcommand{\andd}{\,\,\,\&\,\,\,}

\newcommand{\w}{\omega}
\newcommand{\s}{\sigma}
\newcommand \leb{\lambda}
\newcommand{\Tur}{\textup{\scriptsize T}}

\renewcommand{\le}{\leqslant}
\renewcommand{\ge}{\geqslant}
\renewcommand{\leq}{\leqslant}
\renewcommand{\geq}{\geqslant}
\newcommand{\nle}{\nleqslant}
\renewcommand{\preceq}{\preccurlyeq}
\renewcommand{\succeq}{\succcurlyeq}

\DeclareMathOperator{\High}{High}
\newcommand{\CR}{\textup{CR}}
\newcommand{\MLR}{\textup{MLR}}

\DeclareMathOperator{\id}{id}
\newcommand{\idb}{{\id}^\omega}
\newcommand{\BS}{\omega^\omega}
\DeclareMathOperator{\wt}{wt}

\newcommand{\converge}{\!\!\downarrow}
\newcommand{\diverge}{\!\!\uparrow}

\DeclareMathOperator{\DNC}{DNC}

\newcommand{\PP}{\mathbf{P}}
\newcommand{\TT}{\mathcal A}
\newcommand{\vp}{\varpi}
\newcommand{\vt}{\vartheta}
\newcommand{\vr}{\varrho}
\newcommand{\vs}{\varsigma}

\newcommand{\bi}{\begin{itemize}}
\newcommand{\ei}{\end{itemize}}
\newcommand{\bc}{\begin{center}}
\newcommand{\ec}{\end{center}}

\newcommand{\CF}[1]{\textup{CF}_{#1}}
\newcommand{\AMD}{\textup{AMD}}
\newcommand{\PJD}[1]{\textup{PJD}_{#1}}
\newcommand{\JD}[1]{\textup{JD}_{#1}}
\newcommand{\vphi}{\varphi}
\newcommand{\inst}{\texttt{\textup{inst}}}
\newcommand{\sol}{\texttt{\textup{sol}}}

\newcommand{\sfup}[1]{\textsf{\upshape #1}}

\newcommand{\andre}[1]{{\textcolor{red}{Andre: #1}}  }
\newcommand{\joe}[1]{\textcolor{blue}{Joe: #1}}
\newcommand{\noam}[1]{\textcolor{green}{Noam: #1}}

\renewcommand{\andre}[1]{}
\renewcommand{\joe}[1]{}
\renewcommand{\noam}[1]{}

\begin{document}

\title{Highness properties close to PA-completeness}

\date{\today}

\author[N.\:Greenberg]{Noam Greenberg} 
\address[N.\:Greenberg]{School of Mathematics and Statistics, Victoria University of Wellington, Wellington, New Zealand}
\email{greenberg@sms.vuw.ac.nz}
\urladdr{\url{http://homepages.mcs.vuw.ac.nz/~greenberg/}}

\author[J.\:S.\:Miller]{Joseph S. Miller}
\address[J.\:S.\:Miller]{Department of Mathematics, University of Wisconsin--Madison, 480 Lincoln Dr., Madison, WI 53706, USA}
\email{jmiller@math.wisc.edu}
\urladdr{\url{http://www.math.wisc.edu/~jmiller/}}

\author[A.\:Nies]{Andr\'e Nies}
\address[A.\:Nies]{Department of Computer Science, University of Auckland, Private bag 92019, Auckland, New Zealand}
\email{andre@cs.auckland.ac.nz}
\urladdr{\url{https://www.cs.auckland.ac.nz/~nies/}}

\thanks{This research in this paper was commenced   when the authors participated in the Buenos Aires Semester in Computability, Complexity and Randomness, 2013. Greenberg and Nies were  partially supported by a Marsden grant of the Royal Society of New Zealand. Miller was initially supported by the National Science Foundation under grant DMS-1001847; he is currently supported by grant \#358043 from the Simons Foundation. }
%\hl{Nies was supported by a metallic underwire}
\makeatletter
\@namedef{subjclassname@2010}{\textup{2010} Mathematics Subject Classification}
\makeatother

% 03-XX 			Mathematical logic and foundations
% 	03Dxx 		Computability and recursion theory
% 		03D30   	Other degrees and reducibilities
% 		03D32   	Algorithmic randomness and dimension [See also 68Q30]

% 68-XX 			Computer science
% 	68Qxx 		Theory of computing
% 		68Q30   	Algorithmic information theory (Kolmogorov complexity, etc.) [See also 03D32]

\begin{abstract} Suppose we are given a computably enumerable object arise from algorithmic randomness or  computable analysis. We are interested in 
  the strength of oracles which can compute an object that approximates this  c.e.\      object.   It turns out that,  depending on the type of object, the resulting highness property is either close to,  or equivalent to being PA-complete. We examine, for example, dominating a c.e.\ martingale by an oracle-computable martingale, computing compressions functions for two variants of Kolmogorov complexity, and computing subtrees of positive measure of  a given $\Pi^0_1$ tree of positive measure without dead ends. We prove  a separation result from PA-completeness for the latter property, called the \emph{continuous covering property}. We also separate the corresponding principles in reverse mathematics. 
 % There is an oracle that has the strong continuous covering property and is not PA-complete. \joe{This seems inadequate.} \andre{Will expand. Did you look at the comments below?}
% We identify properties of oracles bringing them close to having PA degree. These properties are based on relativizing concepts from algorithmic information theory, computable analysis, and algorithmic randomness.
\end{abstract}

\maketitle
\tableofcontents

%%%%%%%%
%%%%%%%%
\section{Introduction}  
%%%%%%%%
%%%%%%%%

%
% JOE (6/11/19): This is final
%

Recall that the \emph{PA degrees} are those Turing degrees that can compute a path through every nonempty $\Pi^0_1$ subclass of $2^\omega$, or equivalently, every nonempty, computably bounded $\Pi^0_1$ subclass of $\omega^\omega$. The PA degrees are so named because they are the degrees of complete consistent extensions of Peano Arithmetic. In practice, it is usually easier to think of them as the degrees of DNC$_2$ functions, that is, x diagonally noncomputable, $\{0,1\}$-valued functions. It is easy to see that the collection of all such functions is a $\Pi^0_1$ class in $2^\omega$, and not too difficult to see that all such functions have PA degree. The halting problem, $\emptyset'$, can easily compute a DNC$_2$ function, so it has PA degree. But it is important to understand that PA degrees can be much more computationally feeble, for example, low~\cite{JS:72}.

The PA degrees have played an interesting supporting role in the study of algorithmic randomness, in part because they allow us to approximate certain objects that play a central role. We start with three illustrative examples: plain Kolmogorov complexity, prefix-free complexity, and the optimal supermartingale. All three are intrinsically c.e.\ objects that are optimal in their classes, they are all Turing equivalent to $\emptyset'$, and they all can be approximated using PA degrees. For example, every PA degree computes a martingale that majorizes the optimal c.e.\ supermartingale~\cite{FSY:11}. 

This paper is motivated by the following question: in these and related examples, are the PA degrees necessary? We will see that it depends on the type of example. PA degrees are necessary in the case of plain complexity and the optimal supermartingale, but not in the case of prefix-free complexity. Moving beyond these examples, we will explore the highness class $\High(\CR,\MLR)$, which consists of the oracles relative to which computable randomness implies (unrelativized) Martin-L\"of randomness. Every PA complete set is in this class because, as noted, every PA degree computes a martingale that majorizes the optimal c.e.\ supermartingale. In this case, it remains open whether PA completeness is necessary.  
%%%%%%%%
\subsection*{\texorpdfstring{\boldmath $C$}{C}-compression functions}
%%%%%%%%

Let $C\colon 2^{<\omega}\to\omega$ denote plain Kolmogorov complexity. It is easy to see that~$C$ is computable from $\emptyset'$, and in fact, that they are Turing equivalent. But if all we want to do is find a lower bound for $C$ that respects the same combinatorial restriction as $C$ itself, then it turns out that a PA degree is sufficient. This was first observed by Nies, Stephan, and Terwijn~\cite{NST:05}.

\begin{defn}
A \emph{$C$-compression function} is an injective function $F\colon 2^{<\omega}\to 2^{<\omega}$ such that $|F(\sigma)|\leq C(\sigma)$ for all $\sigma$.	
\end{defn}

Intuitively, we think of $F$ as mapping $\sigma$ to a minimal program for $\sigma$. By definition, $f(\sigma) = |F(\sigma)|$ is a lower bound for $C(\sigma)$ and at most~$2^n$ strings have $f$-complexity~$n$.  %$2^n-1$ strings have $f$-complexity strictly less than $n$. \noam{seems a bit more natural to me?}
Alternately, from an $f$ with these properties, it is easy to compute a $C$-compression function.

The property of being a $C$-compression function is $\Pi^0_1$; if $F$ is not a $C$-compression function, then this is eventually apparent. Furthermore, there is a constant $c$ such that $C(\sigma)\leq |\sigma|+c$, for all $\sigma$, so there are only finitely many possibilities for $F(\sigma)$. In particular, the collection of all $C$-compression functions is a computably bounded $\Pi^0_1$ class. Therefore, every PA degree computes a $C$-compression function. Nies, Stephan, and Terwijn~\cite{NST:05} used this fact, along with the low basis theorem, to prove that every $2$-random has infinitely many initial segments with maximal $C$-complexity (up to a constant). Note that in this case, the PA degrees are used as a tool in proving a result that makes no mention of them.

Kjos-Hanssen, Merkle, and Stephan~\cite[Thm.\ 4.1]{KMS:11} showed that a PA degree is actually necessary to compute a $C$-compression function. In particular, they proved that for some constant $k\in\omega$, there is a uniform procedure that computes a DNC$_k$ function (i.e., a $\{0,\dots,k-1\}$-valued DNC function) from a $C$-compression function. Jockusch~\cite{J:89} showed that DNC$_k$ functions have PA degree. However, he also proved that there is no uniform procedure to compute a DNC$_2$ function from a DNC$_k$ function for $k>2$, so this leaves open a question about uniformity. We prove in Propositions~\ref{prop:C-uniform} and~\ref{prop:C-nonuniform} that the amount of uniformity that is possible depends on the universal (plain) machine that is used to define~$C$. We build a universal machine such that DNC$_2$ functions can be uniformly computed from a $C$-compression function, and another universal machine such that this is impossible.

%%%%%%%%
\subsection*{\texorpdfstring{\boldmath $K$}{K}-compression functions}
%%%%%%%%

Let $K\colon 2^{<\omega}\to\omega$ denote prefix-free (Kolmogorov) complexity. Similar to  $C$, one easily verifies  that the function $K$ is Turing equivalent to~$\emptyset'$. 
\begin{defn}
A \emph{$K$-compression function} is an injective function $F\colon 2^{<\omega}\to 2^{<\omega}$ with prefix-free range such that $|F(\sigma)|\leq K(\sigma)$ for all $\sigma$.	
\end{defn}

This definition ensures that $f(\sigma)=|F(\sigma)|$ is a lower bound for $K(\sigma)$ and, because the range is prefix-free, $\sum_{\sigma\in 2^{<\omega}} 2^{-f(\sigma)} \leq 1$. Conversely, given any such $f$, we can compute a corresponding $K$-compression function $F$ using the Kraft--Chaitin theorem. Note that the property of being a $K$-compression function is $\Pi^0_1$. Furthermore, there is a constant $c$ such that $K(\sigma)\leq 2|\sigma|+c$, for all $\sigma$, so there are only finitely many possibilities for $F(\sigma)$. Thus, as above, the collection of all $K$-compression functions is a computably bounded $\Pi^0_1$ class, so a PA degree can compute a $K$-compression function.  This was observed by Nies~\cite[Solution to Exercise 3.6.16]{N:09} and used by Bienvenu, et al.~\cite[Proposition 3.4 (with Hirschfeldt)]{BGKNT:16}.

In contrast to the previous example, $K$-compression functions are not necessarily PA complete. We show this in Section~\ref{sec:K-compression}. This fact will also follow from Theorem~\ref{thm:main-separation}, but the proof in Section~\ref{sec:K-compression} serves as a nice warm-up for that result.

What are the degrees of $K$-compression functions? We do not have a complete answer, but we can say a couple of things. The $\Pi^0_1$ class of $K$-compression functions was studied by Bienvenu and Porter~\cite[Thm.\ 7.9]{BP:16}, where it was shown that it is \emph{deep} in the sense of their Def.\ 4.1. This implies that no incomplete Martin-L\"of random can compute a $K$-compression function. In Proposition~\ref{prop:sdnc}, we prove that despite the fact that $K$-compression functions do not always compute constant bounded DNC functions, they always compute very slow growing DNC functions: for any computable, nondecreasing, unbounded $h\colon\omega\to\omega\smallsetminus\{0,1\}$, every $K$-compression function computes an $h$-bounded DNC function. Note that for a sufficiently slow growing $h$, the $\Pi^0_1$ class of $h$-bounded DNC functions also is deep~\cite{BP:16}, so this is connected to the previous fact.

%%%%%%%%
\subsection*{Majorizing the optimal supermartingale}
%%%%%%%%

A  c.e.\ supermartingale $m\colon 2^{<\omega}\to \R_{\geq 0}$ is called \emph{optimal} if for each c.e.\ supermartingale $f$ there is a constant $c>0$ such that $cm(\sigma) \ge f(\sigma)$ for each string $\sigma$ (Schnorr; see \cite[Def.\ 5.3.6]{Downey.Hirschfeldt:book}). Obviously, all optimal c.e.\ supermartingales are equal up to multiplicative constants, and so we refer to \emph{the} optimal c.e.\ supermartingale~$m$. Again, $m$ is Turing equivalent to $\emptyset'$ and, again, we can bound~$m$ (this time, from above) using a PA degree.  

\begin{prop}[Franklin, Stephan, and Yu~\cite{FSY:11}]\label{prop:majorizes}
Every PA degree computes a martingale that majorizes the optimal c.e.\ supermartingale.
\end{prop}
\noam{I commented the proof out because it doesn't seem to be used elsewhere.}
\andre{it could go into the arxiv version}

As before, we would like to know if the problem of majorizing~$m$ is PA complete. It is; in Proposition~\ref{prop:mdom}, we construct an atomless c.e.\ martingale $M$ such that only the PA degrees can compute a martingale majorizing~$M$. Since the optimal c.e.\ supermartingale majorizes $M$ up to a multiplicative constant, only the PA degrees can compute a martingale that majorizes~$m$. The proof has an interesting case breakdown that introduces nonuniformity. We show in Proposition~\ref{prop:majorizing-nonuniform} that this nonuniformity is necessary; for no  $k\in\omega$  is  there a uniform procedure (or even a finite collection of procedures) that computes a DNC$_k$ function from every martingale that majorizes $m$. Contrast this with the case of $C$-compression functions. 
%\andre{Can we show that for each $k$, DNC$_k$ is Medvedev above dominating the universal left-c.e. MG?} \joe{We haven't tried to prove this; if not, then it would be a bushy trees proof}

%%%%%%%%
\subsection*{Jordan decomposition}
%%%%%%%%

Our next application of PA degrees is not as obviously related to algorithmic randomness as the first three. Brattka, Miller, and Nies~\cite{BMN:16} used the fact that a PA degree can compute a Jordan decomposition \emph{on the rational numbers} of a computable function of bounded variation. This is connected to the topic of this paper for two reasons. First, it was used in their proof of a result of Demuth~\cite{D:75}: $x\in[0,1]$ is Martin-L\"of random if and only if every computable function $f\colon[0,1]\to\R$ of bounded variation is differentiable at $x$. So it is an application of PA degrees to randomness. Second, we will see in Section~\ref{subsec:Jordan} that the problem of finding a Jordan decomposition on the rationals is closely related to computing a martingale that majorizes an atomless c.e.\ martingale.

If $f\colon[0,1]\to\R$ and $x\in [0,1]$, then the \emph{variation of $f$ on $[0,x]$} is
\[
V_f(x) = \sup_{P}\sum_{i=1}^{n_P} |f(t_i)-f(t_{i-1})|,
\]
where $P$ ranges over finite sequences $0\leq t_0 < t_1 < \cdots < t_{n_P}\leq x$. If $V_f(1)$ is finite, we say that $f$ has \emph{bounded variation}. Jordan proved that a function $f\colon[0,1]\to\R$ has bounded variation if and only if there are nondecreasing functions $g,h\colon[0,1]\to\R$ such that $f = g-h$. In particular, we can take $g = V_f$ and $h = f-g$. Moreover, if $f$ is continuous, then $V_f$ is also continuous, so both $g$ and $h$ can be taken to be continuous.

A PA degree is \emph{not} sufficient to find a continuous Jordan decomposition of a computable function of bounded variation.

\begin{thm}[Greenberg, Nies, and Slaman (unpublished)]\label{thm:Jordan-continuous}
There is a computable function $f\colon [0,1]\to \R$ of bounded variation such that every continuous nondecreasing $g\colon [0,1]\to \R$ for which $g-f$ is also nondecreasing computes $\emptyset'$. 
\end{thm}
\begin{proof}
Fix disjoint rational intervals $I_n$ in $[0,1]$. In each $I_n$, fix an increasing sequence of rational numbers $q^n_k$ converging to $q^n_{\omega} = \max I_n$. Let $I_{n,k} = [q^n_k, q^n_{k+1}]$. Define the function $f$ as follows. If $n$ enters $\emptyset'$ at stage $s$, make $f\rest {I_{n,s}}$ a ``saw-tooth'' with positive variation $2^{-n}$ but of height $2^{-s}$. Elsewhere make $f = 0$. Note that at stage $s$ we can approximate $f$ to within $2^{-s}$, so $f$ is computable. 
	
Suppose that $g$ is a solution. For each $n$, find an $s$ such that $g(q^n_\omega) - g(q^n_s) < 2^{-n}$. Then $n\in \emptyset'$ if and only if $n\in \emptyset'_s$. 	
\end{proof}

Brattka, et al.~\cite{BMN:16} considered a weaker version of Jordan decomposition, one that can be solved using a PA degree. Let $I_\Q = [0,1]\cap\Q$. Given a computable function $f\colon[0,1]\to\R$ of bounded variation, they considered the problem of finding nondecreasing functions $g,h\colon I_\Q\to\R$ such that $f\uh I_\Q = g-h$. We will see that this problem is essentially the same as finding a martingale that majorizes an atomless c.e.\ martingale. This allows us to conclude that a PA degree is necessary, in general, to find Jordan decompositions on the rationals. Furthermore, the connection to martingales implies that we cannot uniformly compute DNC$_k$ functions, for any $k$, from the Jordan decompositions of a computable function of bounded variation.

%%%%%%%%
\subsection*{\texorpdfstring{\boldmath $\High(\CR,\MLR)$}{High(CR,MLR)}}
%%%%%%%%

We turn to an application of PA degrees to randomness where the relationship to PA-completeness remains open. We say that an oracle $D$ is \emph{high for computable randomness versus Martin-L\"of randomness} if whenever a sequence is computably random relative to $D$, it is Martin-L\"of random. We abbreviate this property as $\High(\CR,\MLR)$ and use the same notation for the collection of such oracles. Franklin, Stephan, and Yu~\cite{FSY:11} were the first to study highness for pairs of randomness notions, and in particular, were the first to study the class $\High(\CR,\MLR)$.\footnote{Lowness for pairs of randomness notions was introduced earlier by Kjos-Hanssen, Nies, and Stephan \cite{KNS:05}.} They observed that if $D$ has PA degree, then it is $\High(\CR,\MLR)$. To see this, note that by Proposition~\ref{prop:majorizes}, there is a $D$-computable martingale $M$ that majorizes the optimal c.e.\ supermartingale. In particular, $M$ succeeds on every non-ML-random sequence, so no non-ML-random can be computably random relative to $D$.

\begin{prop}[Franklin, Stephan, and Yu~\cite{FSY:11}]\label{prop:PA-high}
Every PA-complete oracle is $\High(\CR,\MLR)$.
\end{prop}

Most similar highness classes and lowness classes for pairs of randomness notions are well-understood. In contrast, we cannot answer a very fundamental question about $\High(\CR,\MLR)$: is every $\High(\CR,\MLR)$ oracle PA-complete? Some things are known. Franklin, et al.~\cite{FSY:11} showed that every $D\in \High(\CR,\MLR)$ computes a Martin-L\"of random, and that the class $\High(\CR,\MLR)$ has measure zero. Note that if $D$ is $\High(\CR,\MLR)$, we know that all $D$-computable martingales together succeed on the non-ML-random sequences. In fact, we can do better:

% 2009
\begin{prop}[Kastermans, Lempp, and Miller; see Bienvenu and Miller~{\cite[Proposition 20]{BM:12}}]\label{prop:single-martingale}
If $D$ is $\High(\CR,\MLR)$, then there is a \emph{single} $D$-computable martingale $N$ that succeeds on every non-ML-random.
\end{prop}

Note that $N$ need not majorize the optimal c.e.\ supermartingale $m$; it must only succeed on every sequence on which $m$ succeeds.

The class $\High(\CR,\MLR)$ also appears in unpublished work of Miller, Ng, and Rupprecht. Building on the proposition above, they proved that there is a single $D$-computable martingale that succeeds on every non-computably random sequence if and only if $D$ is $\High(\CR,\MLR)$ or high (i.e., $D'\geq_\Tur\emptyset''$). The same oracles are necessary to compute a single martingale that succeeds on every non-Schnorr random sequence. Finally, $D$ computes a single martingale that succeeds on all non-Kurtz random sequences if and only if $D$ is $\High(\CR,\MLR)$ or has hyperimmune degree.

%%%%%%%%
\subsection*{Covering properties}
%%%%%%%%

In attempting to capture computability-theoretic properties weaker than---but close to---PA, we introduce two ``covering properties''. Both properties say, in different ways, that an oracle can compute a ``small'' cover of a ``small'' c.e.\ object.

For the first,  if $\bar A = \seq{A_n}$ is a sequence of subsets of $\w$, we let $\wt(\bar A) = \sum_n 2^{-n}|A_n|$. In other words, every element of $A_n$ receives weight $2^{-n}$. We say that $\bar B$ \emph{covers} $\bar A$ if $A_n\subseteq B_n$ for all $n$. An oracle $D$ has the \emph{discrete covering property} if every uniformly c.e.\ sequence of finite weight is covered by some $D$-computable sequence of finite weight. In Section~\ref{sec:discrete}, we prove that an oracle $D$ has the discrete covering property if and only if it computes a $K$-compression function. We also prove that such an oracle computes slow growing DNC functions: for any order function $h\colon\omega\to\omega\smallsetminus \{0,1\}$, there is an $h$-bounded DNC function computable from $D$.

\medskip
For our second covering property, if $U\subseteq 2^{\w}$ is open, then let $S_U$ be the set of strings $\s$ such that $[\s]\subseteq U$. We say that an oracle $D$ has the \emph{continuous covering property} if for every $\Sigma^0_1$ class $U\subseteq 2^\w$ of measure less than $1$, there is an open superset $V\supseteq U$ such that $\leb(V)<1$ and $S_V$ is computable from $D$. Equivalently, by focusing on the complements of the open sets, $D$ has the continuous covering property if for any computable tree $T$ such that $\leb([T])>0$, there is a $D$-computable tree $S$ with no dead ends such that $S\subseteq T$ and $\leb([S])>0$. If we further require that every nonempty, relatively clopen subtree of~$S$ has positive measure, we get a useful variant: the \emph{strong continuous covering property}.

In Section~\ref{sec:continuous}, we prove that the (strong) continuous covering property is implied by $\High(\CR,\MLR)$ and that it implies the discrete covering property. We do not know if either implication is strict. However, in Theorem~\ref{thm:main-separation}, we prove that having the strong continuous covering property is strictly weaker than having PA degree; this is the most difficult argument  of the paper.

Recall that Franklin, et al.~\cite{FSY:11} showed that every oracle in $\High(\CR,\MLR)$ computes a Martin-L\"of random. This also holds for an oracle $D$ with the continuous covering property. To see this, take a $\Sigma^0_1$ class $U\subseteq 2^\omega$ such that $\leb(U)<1$ and all non-ML-random sequences are in $U$. For example, $U$ could be the second  component of   a universal Martin-L\"of test. Let $T$ be a computable tree such that $[T] = 2^\omega\smallsetminus U$. Take a $D$-computable subtree $S\subseteq T$ with no dead ends such that $\leb([S])>0$. Of course, every infinite path through $S$ in Martin-L\"of random, and since $S$ has no dead ends, there are paths computable from $S\leq_\Tur D$. We have proved:

\begin{prop}\label{prop:continuous-ML}
Every oracle with the continuous covering property computes a Martin-L\"of random.
\end{prop}

% \begin{center}
% \begin{tikzpicture}[text centered, node distance=0.75cm, auto]
% 	\node (pa) {PA};
% 	\node[below=of pa] (h) {$\High(\CR,\MLR)$};
% 	\node[below=of h, text width=3.05cm] (cc) {continuous covering\\property};
% 	\node[below=of cc, text width=2.6cm] (dc) {discrete measure\\property};
% 	\node[left=of dc, text width=3.7cm] (Kc) {computes a\\$K$-compression function};
% 	\node[below=of dc, text width=2.3cm] (sdnc) {computes slow\\growing DNC};
% 	\node[right=of sdnc, text width=2.1cm] (mlr) {computes a ML-random};

% 	\path (pa) edge[imp] (h)
% 		(h) edge[imp] (cc)
% 		(cc) edge[imp] (dc)
% 		(cc.east) edge[imp,bend left] (mlr)
% 		(dc) edge[imp,<->] (Kc)
% 		(dc) edge[imp] (sdnc);
		
% 	\path (mlr) edge[imp,<->] node[sloped,yshift=8pt,xshift=-4pt] {$\not$} (sdnc)
% 		(cc.west) edge[imp,bend left=100] node[sloped,yshift=-8pt,xshift=-10pt] {$\not$} (pa.west);

% 	\path (h) edge[imp,bend right] node[right] {\bf\Large ?} (pa)
% 		(cc) edge[imp,bend right] node[right] {?} (h)
% 		(dc) edge[imp,bend right] node[right] {?} (cc)
% 		(dc) edge[imp] node[right,yshift=4pt] {?} (mlr);
% \end{tikzpicture}
% \end{center}

\begin{figure}[t]\label{fig:Muchnik}
\begin{tikzpicture}[text centered, node distance=0.75cm, auto]
	\node (pa) {PA};
	\node[below=of pa] (h) {$\High(\CR,\MLR)$};
	\node[below=of h, text width=3.05cm] (cc) {(strong) continuous\\covering property};
	\node[below=of cc, text width=2.6cm] (dc) {discrete covering\\property};
	\node[left=1.5cm of dc, text width=2.4cm] (Kc) {computes a\\$K$-compression\\function};
	\node[below=of dc, text width=2.6cm] (sdnc) {computes a slow\\growing DNC};
	\node[right=of sdnc, text width=2.1cm] (mlr) {computes a ML-random};

	\path (pa) edge[imp] node[left] {\tiny Prop.~\ref{prop:PA-high}} (h)
		(h) edge[imp] node[left] {\tiny Prop.~\ref{prop:high-continuous}} (cc)
		(cc) edge[imp] node[left] {\tiny Prop.~\ref{prop:continuous-to-discete}} (dc)
		(cc.east) edge[imp,bend left] node[right] {\tiny Prop.~\ref{prop:continuous-ML}} (mlr)
		(dc) edge[imp,<->] node[above] {\tiny Prop.~\ref{prop:equiv}} (Kc)
		(dc) edge[imp] node[left] {\tiny Prop.~\ref{prop:sdnc}} (sdnc);
		
	\path (mlr) edge[imp,<->] node[sloped,yshift=8pt,xshift=-4pt] {$\not$} (sdnc)
		(cc.west) edge[imp,bend left=90] node[sloped,yshift=-8pt,xshift=-10pt] {$\not$} node[left] {\tiny Thm.~\ref{thm:main-separation}} (pa.west);

	\path (h) edge[imp,bend right] node[right] {?} (pa)
		(cc) edge[imp,bend right] node[right] {?} (h)
		(dc) edge[imp,bend right] node[right] {?} (cc)
		(dc) edge[imp] node[right,yshift=4pt] {?} (mlr);
\end{tikzpicture}
\caption{The relationship of the covering properties to other computability-theoretic properties.}
\end{figure}
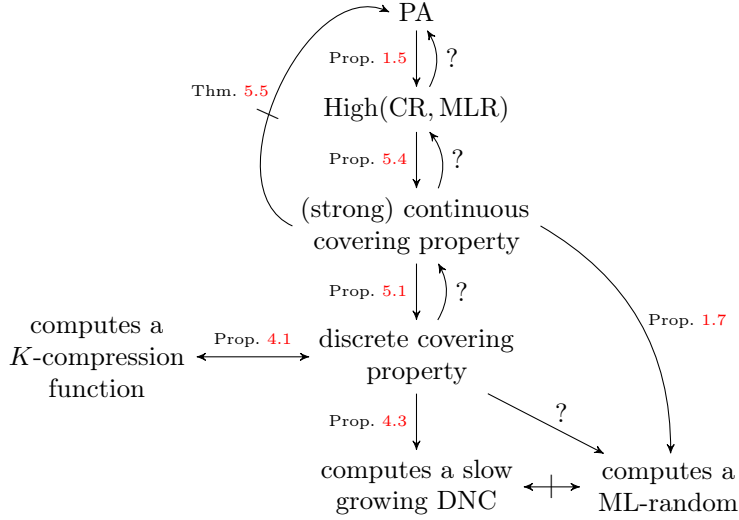

Fig.~\ref{fig:Muchnik} gives a summary of the results related to the covering properties and indicates several open question. Two references are missing from the diagram. The fact that a Martin-L\"of random sequence need not compute a slow growing DNC function was mentioned above; it follows from Bienvenu and Porter's~\cite{BP:16} work on deep $\Pi^0_1$ classes. Greenberg and Miller~\cite{GM:11} proved that slow growing DNC functions do not always compute a Martin-L\"of random. That is, they proved that if $h\colon\omega\to\omega\smallsetminus\{0,1\}$ is any order function (i.e., a computable, nondecreasing, unbounded function), then there is an $h$-bounded DNC function that does not compute a Martin-L\"of random. (A related but different proof of this fact was given by Khan and Miller~\cite{KM:17}.)

%%%%%%%%
\subsection*{Strong weak weak K{\H o}nig's lemma}
%%%%%%%%

In the final section of the paper, we consider   reverse mathematical aspects of some of our results. In particular, we introduce a new principle, \emph{strong weak weak K{\H o}nig's lemma} ({\sffamily\itshape SWWKL}), which corresponds to the strong continuous covering property. It states:

\medskip
\hangindent=30pt
\noindent\hspace{30pt}%
If $T\subseteq 2^{<\omega}$ is a tree with positive measure, then there is a nonempty subtree $S\subseteq T$ such that if $\sigma\in S$, then $S$ has positive measure above $\sigma$.

\medskip
\noindent We also introduce the apparently weaker principle, \emph{weak strong weak weak K{\H o}nig's lemma} ({\sffamily\itshape WSWWKL}), which corresponds to the continuous covering property; it only requires that $S\subseteq T$ has positive measure and no dead ends, so unlike \sfup{SWWKL}, it does not even guarantee that $S$ is perfect. Using the work of the previous sections, we prove that, over $\sfup{RCA}_0$, both \sfup{SWWKL} and \sfup{WSWWKL} are strictly between weak K{\H o}nig's lemma (\sfup{WKL}, the axiom corresponding to the existence of sets of PA degree) and weak weak K{\H o}nig's lemma (\sfup{WWKL}, the axiom corresponding to the existence of Martin-L\"of random sequences).

%%%%%%%%
\subsection*{Other uses of PA degrees in algorithmic randomness}
%%%%%%%%

We started the introduction by declaring that PA degrees play an interesting supporting role in algorithmic randomness. We have presented some evidence for this claim, but the reader should not be led to believe that we have exhausted the subject. Far from it. Without making a complete survey, let us finish the introduction by mentioning a few other examples.  
\smallskip
\begin{itemize}[itemsep=4pt,leftmargin=*]
\item Stephan~\cite{S:06} proved that a Martin-L\"of random sequence has PA degree if and only if it computes $\emptyset'$; this gives an easy proof of a result of Ku{\v c}era~\cite{K:85}, that the degrees of Martin-L\"of random are not closed upward.

\item Barmpalias, Lewis, and Ng~\cite{BLN:10} proved that every PA degree is the join of two Martin-L\"of random degrees. Together with the previous result, this gives many examples of pairs of random degrees that join to a nonrandom degree.

\item Stephan and Simpson~\cite{SS:15} gave an unexpected characterization of the $K$-trivial sets (i.e., those with minimal growth of prefix-free Kolmogorov complexity of their initial segments) as the sets computable from every PA degree relative to which Chaitin's $\Omega$ remains ML-random.

\item Higuchi, Hudelson, Simpson, and Yokoyama~\cite{HHSY:14} proved that strong $f$-randomness is equivalent to $f$-randomness relative to a PA degree. Both $f$-randomness and its strong variant are ``partial randomness'' are notions that have been studied, in various degrees of generality, by several authors.
\end{itemize}

%%%%%%%%
\subsection*{Notation}
%%%%%%%%

For $f,h\in\BS$ we write $f\leq h$ to mean that $f$ is majorized by $h$. In this case, we say that $f$ is \emph{$h$-bounded}. Let $h^\omega = \set{f\in\omega^\omega}{f\leq h}$ be the set of $h$-bounded functions. We write $h^{<\omega}$ for the set $\set{\sigma\in \omega^{<\omega}}{(\forall n<|\sigma|)\; \sigma(n)\leq h(n)}$ of $h$-bounded strings. Let $\idb = \{f\in\BS \colon (\forall n)\; f(n)\leq n\}$, in other words, the \emph{identity bounded} functions. 

We let~$J$ denote a fixed universal partial computable function, based on an acceptable listing of the partial computable functions; a common choice is $J(e)= \varphi_e(e)$. A function~$f$ is diagonally non-computable if $J(e)\ne f(e)$ whenever $J(e)\converge$.

%%%%%%%%
%%%%%%%%
\section{Properties   that imply PA degree}
%%%%%%%%
%%%%%%%%

In this section, we look at three examples  from algorithmic randomness where PA degrees turn out to be necessary. We will see that $C$-compression functions and martingales that majorize the optimal c.e.\ supermartingale must have PA degree. 

We will also show that there is a computable function of bounded variation $f\colon[0,1]\to\R$ such that every Jordan decomposition of~$f$ on the rationals has PA degree. In each case, we examine the amount of uniformity possible.

%%%%%%%%
\subsection{\texorpdfstring{\boldmath $C$}{C}-compression functions}
%%%%%%%

Kjos-Hanssen, Merkle, and Stephan~\cite{KMS:11} gave a uniform procedure to compute a DNC$_k$ function from a $C$-compression function, for some $k$. Since we shortly use it, for completeness, we reproduce their proof.

\begin{prop}[Kjos-Hanssen, Merkle, and Stephan~{\cite[Theorem~4.1]{KMS:11}}] \label{prop:KHMS}
Every $C$-compression function has PA degree. Moreover, for large enough $k\in\omega$, there is a uniform way to compute a DNC$_k$ function from a $C$-compression function.
\end{prop}
\begin{proof}
Define a partial computable function $\psi\colon 2^{<\omega}\to 2^{<\omega}$ as follows: if $\sigma\in 2^{<\omega}$ has length $n$ and if $J(n)\downarrow = \tau$ (where we view $\tau$ as an element of $2^{<\omega}$), then $\psi(\sigma) = \sigma\conc \tau$. Note that there is a constant $c\in\omega$ such that $C(\psi(\sigma)) < |\sigma| + c$.

Let $F\colon 2^{<\omega}\to 2^{<\omega}$ be a $C$-compression function. From such an $F$, we can uniformly compute a function $f\colon\omega\to 2^{<\omega}$ such that $|f(n)|=n$ and $C(f(n))\geq n$. Consider the $F$-computable function $g\colon\omega\to 2^{<\omega}$ such that $g(n)$ is the last $c$ bits of $f(n+c)$. We claim that $g$ is DNC$_k$, where $k = 2^c$. Assume that $g(n) = J(n)$. Let $\sigma$ be the length $n$ prefix of $f(n+c)$. Then $C(f(n+c)) = C(\psi(\sigma)) < |\sigma| + c = n+c$, which is a contradiction.
\end{proof}

As we said in the introduction, the $k$ in the previous result depends on the universal machine used to define $C$. By designing our machine for the purpose, we can ensure that there is a uniform procedure to compute a DNC$_2$ function from a $C$-compression function; this is the next result. On the other hand, in Proposition~\ref{prop:C-nonuniform}, we give a universal plain machine for which this fails.

\begin{prop}\label{prop:C-uniform}
There is a universal plain machine $V\colon 2^{<\omega}\to 2^{<\omega}$ such that there is a uniform way to compute a DNC$_2$ function from a $C_V$-compression function.
\end{prop}
\begin{proof}
We modify the proof of the previous proposition. Now let $\psi\colon 2^{<\omega}\to 2^{<\omega}$ be the partial computable function defined as follows: if $\sigma\in 2^{<\omega}$ has length $2n+1$ and if $\phi_n(n)\downarrow = \tau$, then $\psi(\sigma) = \sigma\conc \tau$.

We want $V$ to be a universal plain machine such that $C_V(\psi(\sigma)) < |\sigma| + 1$. Let $\widehat{V}$ be a given universal machine and define $V(00\conc \rho) = \widehat{V}(\rho)$ for all $\rho\in 2^{<\omega}$. This ensures that $V$ is universal, while leaving \nicefrac{3}{4} of the strings of each length free to be used otherwise. Note that $\nicefrac{3}{4}\cdot 2^{2n+1} + \nicefrac{3}{4}\cdot 2^{2n} > 2^{2n+1}$. So we have room left in the domain of $V$ to ensure that $C_V(\psi(\sigma)) \leq |\sigma|$ for every (odd length) $\sigma\in 2^{<\omega}$.

As before, from a $C_V$-compression function $F$, we can uniformly compute an $f\colon\omega\to 2^{<\omega}$ such that $|f(n)|=n$ and $C_V(f(n))\geq n$. Consider the $F$-computable function $g\colon\omega\to 2^{<\omega}$ such that $g(n)$ is the last bit of $f(2n+2)$. We claim that $g$ is DNC$_2$. If not, then $g(n) = \phi_n(n)$ for some $n\in\omega$. Let $\sigma$ be the length $2n+1$ prefix of $f(2n+2)$. Then $C_V(f(2n+2)) = C_V(\psi(\sigma)) \leq |\sigma| = 2n+1$, which is a contradiction.
\end{proof}

Toward proving Proposition~\ref{prop:C-nonuniform}, we need to the following simple combinatorial lemma. It generalizes the observation that either a graph $G$ (on at least two vertices) or its complement $\overline G$ has no isolated vertices: if $v$ is isolated in $G$, then it has edges to every other vertex in $\overline G$. Recall that for any set $X$, the set of subsets of $X$ of size $k$ is written $[X]^k$.

\begin{lem}\label{lem:DumbRT}
Let $X$ be an arbitrary set and fix $k\in\omega$. For any colouring $c\colon[X]^k\to k$, there is an $i<k$ such that
\[
(\forall v\in X)(\exists w_1,\dots, w_{k-1}\in X)\; c(\{v,w_1,\dots,w_{k-1}\}) = i.
\]
\end{lem}
\begin{proof}
We prove this lemma by induction on $k$. Note that it is trivial for $k=1$. Now assume that it holds for $k$ and consider a colouring $c\colon[X]^{k+1}\to {k+1}$.

If the lemma holds for $i=k$, we are done. Otherwise, there is a $u\in X$ such that the induced colouring $\hat{c}$ on $[X\smallsetminus\{u\}]^k$ has range in $k = \{0,\dots,k-1\}$. Hence, by induction, there is an $i<k$ such that
\[
(\forall v\in X\smallsetminus\{u\})(\exists w_1,\dots, w_{k-1}\in X\smallsetminus\{u\})\; \hat{c}(\{v,w_1,\dots,w_{k-1}\}) = i.
\]
But then, for all $v\in X\smallsetminus\{u\}$, we have
\[
(\exists w_1,\dots, w_{k-1}\in X\smallsetminus\{u\})\; c(\{v,u,w_1,\dots,w_{k-1}\}) = i. \qedhere
\]
\end{proof}

The lemma fails if we increase the number of colours. To see this, let $X = \{0,\dots,k\}$ and define $c\colon[X]^k\to k+1$ by $c(\{0,\dots,i-1,i+1,\dots,k\}) = i$. Then for every $i<k+1$, if $Y\subseteq X$ has size $k$ and $i\in Y$, then $c(Y)\neq i$.

\smallskip
The proof of the following proposition illustrates a technique that will be used in later proofs. We want to diagonalize against a functional~$\Gamma$ on some element of a $\Pi^0_1$ class $P$. However, we are not able to effectively guarantee that any \emph{specific} string $\sigma$ has an extension in $P$. Our solution is to use the structure of $P$ together with Lemma~\ref{lem:DumbRT} to diagonalize against~$\Gamma$ on enough strings so that we know that at least one of them is extendible in $P$.

\begin{prop}\label{prop:C-nonuniform}
For each $k$, there is a universal plain machine $V$ such that there is no uniform way to compute a DNC$_k$ function from a $C_V$-compression function.
\end{prop}

In fact, we show the following. For a plain machine~$V$ and $k\in \omega$, we say that~$V$ \emph{uses at most $\nicefrac{1}{k}$ of the available strings of each length} if for all~$n$, There are at most $2^n/k$ many strings of length~$n$ in the domain of~$V$. 

\begin{prop}\label{prop:C-nonuniform:restatement}
Let $k\in \omega$. If~$V$ is a universal plain machine which uses at most $\nicefrac{1}{k}$ of the available strings of each length, then there is no uniform way to compute a DNC$_k$ function from a $C_V$-compression function.
\end{prop}

Proposition~\ref{prop:C-nonuniform} follows by taking any universal plain machine~$U$, fixing some~$c \ge \log_2 k$, and letting $V(0^c\s) = U(\s)$ for all~$\s$. 

\begin{proof}[Proof of Proposition~\ref{prop:C-nonuniform:restatement}]
Let $P$ be the collection of $C_V$-compression functions. By a fixed computable numbering of all finite binary strings, we can view~$P$ as a $\Pi^0_1$ class in Baire space; as mentioned above,~$P$ is computably bounded. Fix a computable function $h\colon\omega\to\omega$ such that $P\subseteq h^\omega$. 
%Fix a computable tree $T\subseteq h^{<\w}$ such that $P$ is the set of paths through~$T$ (we write $P = [T]$). 
% We may assume that every $\s\in T$ is injective. 
By the assumption on~$V$, there is some $F\in P$ such that the range of~$F$ includes at most $\nicefrac{1}{k}$ of the strings of each length.

Now assume, for a contradiction, that $\Gamma$ computes a DNC$_k$ function from every $F\in P$. Without loss of generality, $\Gamma$ is $k$-valued. We may also make it total on~$h^\w$ by ensuring that it converges on any oracle not in~$P$.

We define a computable process that will output an $i<k$. The result of this process will be $J(e)$, for some $e$. By the recursion theorem, we may assume that we know $e$ in advance.\footnote{More formally, we define a partial computable function~$\psi$; by the recursion theorem, there is some~$e$ such that $J(e)= \psi(e)$.} By compactness, there is an $n\in\omega$ such that $\Gamma(\sigma,e)\downarrow$ for every $\sigma\in h^n$ (where recall that $h^n$ is the collection of $h$-bounded strings of length~$n$). Our goal is to output an $i<k$ such that $\Gamma(\sigma,e) = i$ for some $\sigma\in h^n$ that is extendible to an element of $P$. Of course, we cannot hope to effectively identify such a $\sigma$, but we will see that we \emph{can} effectively find such an $i$.

\smallskip

Define~$E$ to be the collection of strings $\s\in h^n$ which are injective; and let
\[
\widehat{E} = \set{\tau\in E}{\tau\text{ maps onto at most $\nicefrac{1}{k}$ of the strings of each length}}.
\]
By our assumption on~$V$, we know that there is a $\tau\in \widehat{E}$ that is extendible to an element of~$P$.

We define a map $Q\mapsto \tau_Q$ from $[\widehat{E}]^k$ to~$E$ as follows: given $Q\subseteq \widehat{E}$ of size~$k$, since each $\s\in Q$ maps to at most $\nicefrac{1}{k}$ of the strings of each length, we can let $\tau_Q$ map each $x<n$ to a string $\tau_Q(x)$ with $|\tau_Q(x)|\le |\s(x)|$ for all $\s\in Q$. That is, there is enough room in the range to permit all of the desired compression while keeping~$\tau_Q$ injective. Note that the minimality condition also implies that $\tau_Q\in h^n$, and so $\tau_Q\in E$. What is important is that if some $\s\in Q$ is exetendible to an element of~$P$, then so is~$\tau_Q$. 

Now define a colouring $c\colon[\widehat{E}]^k\to k$ as follows: for $Q\in[\widehat{E}]^k$, let $c(Q) = \Gamma(\tau_Q,e)$. Fix $i<k$ as in Lemma~\ref{lem:DumbRT} for the colouring~$c$; this is the output of our computable procedure, i.e., $J(e)=i$. Now fix $\s\in\widehat{E}$ extendible to an element of $P$ and any $Q\in[\widehat{E}]^k$ such that $\s\in Q$ and $c(\tau_Q) = i$. Then $\tau_Q$ is extendible to an element of $P$, but $\Gamma(\tau_Q,e) = i = J(e)$, which contradicts our choice of~$\Gamma$.
\end{proof}

Recall that for sets $P,R\subseteq \w^\w$, we write $P\le_s R$ (and say that $P$ is \emph{Medvedev reducible} to~$R$) if there is a Turing functional~$\Gamma$ such that for all $X\in R$, $\Gamma(X)$ is total and $\Gamma(X)\in P$: each element of~$R$ computes an element of~$P$, uniformly. In contrast, $P\le_w R$ ($P$ is \emph{Muchnik reducible} to~$R$) if every element of~$R$ computes an element of~$P$, but not necessarily uniformly. Jockush's result mentioned above shows that the classes DNC$_k$ are all Muchnik equivalent (their upward closures in the Turing degrees consist of the PA-complete oracles), but that for all~$k$, $\textup{DNC}_{k+1} <_s \textup{DNC}_k$. The class $\DNC_2$ is Medvedev-complete for computably bounded $\Pi^0_1$ classes. 

For a universal plain machine~$V$, let $\CF{V}$ be the collection of $C_V$-compression functions. Kjos-Hanssen, Merkle, and Stephan's Proposition~\ref{prop:KHMS} says that for every universal plain machine~$V$ there is some~$k$ such that $\DNC_k \le_s \CF{V}$; it follows that for all~$V$, each $\CF{V}$ is Muchnik equivalent to $\DNC_2$. Proposition~\ref{prop:C-uniform} states that for some universal~$V$, $\DNC_2\le_s \CF{V}$; Proposition~\ref{prop:C-nonuniform} says that for every~$k$ there is some~$V$ for which $\DNC_k\nle_s \CF{V}$. 

\smallskip

The proof of Proposition~\ref{prop:C-nonuniform:restatement} can be restricted above any extendible string~$\s$. This is relevant to the following. 

\begin{lem} \label{lem:general_prop_on_finitely_many_functionals}
	Let $k\in \omega$ and let~$P$ be a computably bounded $\Pi^0_1$ class. Suppose that for every $\s$ extendible on~$P$, $\DNC_k \nle_s P\cap [\s]$. Then there is no finite collection $\Gamma_1, \dots, \Gamma_m$ of functionals such that for all $X\in P$, $\Gamma_i(X)\in \DNC_k$ for some $i\le m$. 
\end{lem}

That is, not only do elements of~$P$ not compute $\DNC_k$ functions uniformly, but no finite collection of functionals is sufficient for $\DNC_k \le_w P$. We remark that the proof of Lemma~\ref{lem:general_prop_on_finitely_many_functionals} only uses the fact that $\DNC_k$ is a $\Pi^0_1$ class which is determined pointwise, entry by entry; thus, for example, it also applies to separating classes. 

\begin{proof}
	For brevity, in this proof, for $m\ge 1$ and $\Pi^0_1$ classes $P$ and~$Q$, write $Q\le_m P$ if there is a collection $\Gamma_1,\dots, \Gamma_m$ of $m$-many functionals which together reduce~$Q$ to~$P$, that is, for all $X\in P$, $\Gamma_i(X)\in Q$ for some $i\le m$. 

	By induction on~$m$, we show that for all~$\s$ which is extendible on~$P$, $\DNC_k \nle_m P\cap [\s]$. The case $m=1$ is the assumption of the proposition. 

	Let $m>1$ and suppose that this has been proved for $m-1$. Let $\Gamma_1,\dots, \Gamma_m$ be a collection of $m$-many functionals. Define a functional~$\Theta$ as follows: for all~$X$ and~$e$, $\Theta(X,e) = \Gamma_i(X,e)$ for the first~$i$ for which we see the convergence (if there is such). Let~$\s$ be extendible on~$P$. By assumption, we know that~$\Theta$ cannot witness that $\DNC_k\le_1 P\cap [\s]$. If $\Theta$ is not total on $P\cap [\s]$ then we are done. Otherwise, there is some $\tau\succeq \s$, extendible on~$P$, and some~$e$, such that $\Theta(\tau,e)\converge = J(e)$. There is some~$i$ such that $\Theta(\tau,e) = \Gamma_i(\tau,e)$. Now apply the induction hypothesis to the collection of functionals $\{\Gamma_j\,:\, j\le m, j\ne i \}$ and~$\tau$ to see that this collection cannot witness $\DNC_k \le_{m-1} P\cap [\tau]$; it follows that the original collection $\Gamma_1,\dots, \Gamma_m$ cannot witness $\DNC_k \le_m P\cap [\s]$ either. 
\end{proof}

As mentoned, the proof of Proposition~\ref{prop:C-nonuniform:restatement} gives the assumption of Lemma~\ref{lem:general_prop_on_finitely_many_functionals}, and so we get:

\begin{prop}
For each $k$, there is a universal plain machine $V$ such that there is no finite collection of functionals $\Gamma_1, \dots, \Gamma_m$ so that if $F$ is a $C_V$-compression function, then at least one of $\Gamma_1(F), \dots, \Gamma_m(F)$ is a $\DNC_k$ function.
\end{prop}

%%%%%%%%
\subsection{Majorizing the optimal c.e.\ supermartingale}
%%%%%%%%

The case of martingales that majorize the optimal c.e.\ supermartingale is somewhat different from that of $C$-compression functions. Although each such martingale has PA degree, the proof has an unusual case breakdown that precludes uniformity. We will see in Proposition~\ref{prop:majorizing-nonuniform} that this nonuniformity is necessary: for all $k$, there is no uniform way to compute a $\DNC_k$ function from a martingale majorizing the optimal c.e.\ supermartingale. 
%In fact, we will show that infinitely many reductions are necessary. 
\noam{Removed ``In fact, we will show that infinitely many reductions are necessary.''}

% In the rest of this section, we seamlessly move between a martingale~$M$ and the associated Borel measure~$\mu$ determined by $\mu(\s) = 2^{-|\s|} M(\s)$. A martingale~$N$ majorizes~$M$ if and only if the associated measure $\mu_N$ majorizes~$\mu_M$.

\begin{prop}\label{prop:mdom}
There is an (atomless) c.e.\ martingale $M$ such that every martingale majorizing $M$ has PA degree.	
\end{prop}
\begin{proof}
Define a c.e.\ martingale~$M$ as follows. If $n$ enters $\emptyset'$ at stage $s$, find a string $\sigma\in 2^s$ that looks DNC$_2$ at stage $s$, add $2^{-n}$ much capital to the root and push it up to~$\s$.\footnote{In other words, for $\tau\preceq\sigma$, we let $M_{s+1}(\tau)-M_s(\tau) = 2^{|\tau|-n}$; to preserve the martingale property, for $\tau\succ\sigma$, we let $M_{s+1}(\tau)-M_s(\tau) = 2^{|\sigma|-n}$.}

Now let $N$ be a martingale that majorizes $M$.

\emph{Case 1.} $N$ has a DNC$_2$ atom.\footnote{That is, the associated measure $\mu(\s)= 2^{-|\s|}M(\s)$ has a $\DNC_2$ atom.} A martingale computes all of its atoms, so in this case, $N$ has PA degree.

\emph{Case 2.} $N$ has no DNC$_2$ atoms. Then for each~$n$, there is a stage $f(n)=s$ such that for all strings~$\s$ of length~$s$ that still look DNC$_2$ at stage $s$ we have $N(\s)\le 2^{s-n}$. By construction, $f$, which is~$N$-computable, majorizes the settling time function for~$\emptyset'$, so~$N$ has PA degree.
\end{proof}

Of course, the optimal c.e.\ supermartingale majorizes $M$, up to a multiplicative constant, so we have the desired result:

\begin{cor} \label{cor:majorizing_optimal_supermartingale}
Every martingale that majorizes the optimal c.e.\ supermartingale has PA degree.
\end{cor}

\begin{remark} \label{rmk:continuous_degrees}
	The statements of Proposition~\ref{prop:mdom} and Corollary~\ref{cor:majorizing_optimal_supermartingale} are imprecise. The reason is that objects such as martingales (and below, real-valued functions on the rationals) do not necessarily have Turing degree. Rather, they have a continuous degree (\cite{Miller:Continuous}). Continuous reducibility uses the notion of a \emph{name} of an object. For example, a name of a martingale~$N$ is a function taking a string~$\s$ and a positive rational number~$\epsilon$ to a rational number~$q$ satisfying $|M(\s)-q|<\epsilon$. If~$x$ and~$y$ are objects which have continuous degree (points in computable metric spaces), then $x\le_r y$ if every name for~$y$ computes a name for~$x$. This extends Turing reducibility. The proof of Proposition~\ref{prop:mdom} shows that if~$N$ is a martingale dominating~$M$ then the continuous degree of~$N$ lies above a PA-complete Turing degree. 

	Below, however, we will need to relax this notion of reducibility. Corollary~\ref{cor:majorizing_optimal_supermartingale} is really intended as a statement about Turing degrees: if a set~$X$ can compute a martingale~$N$ dominating~$m$, then~$X$ is PA-complete. To show that it suffices to show that if~$N$ dominates~$m$ then every name for~$N$ computes a $\DNC_2$ function; it is not required that every name for~$N$ computes the \emph{same} $\DNC_2$ function. This is important when we consider uniformity, in particular Weihrauch reducibility, below. 
\end{remark}

% In the lagnuage of reducibilities, we see that $\DNC_2$ and the collection of martingales dominating the optimal c.e.\ supermartingale are Muchnik equivalent. We now show that they are not Medvedev equivalent:

\begin{prop}\label{prop:majorizing-nonuniform}
It is not possible to uniformly compute a $\DNC_k$ function from a martingale majorizing the optimal c.e.\ supermartingale. 
% Moreover, no finite collection of functionals is sufficient.
\end{prop}

More precisely, it is not possible to uniformly compute a $\DNC_k$ function from a name for such a martingale, even if the reduction procedure does not promise to compute the same $\DNC_k$ function from all names for the same martingale. In other words, $\DNC_k$ is not Medvedev below the collection of names for martingales majorizing~$m$.

\noam{The following proof is coarse... but I don't see how to improve it. Once you fix the initial capital (even say within a small band of possible values), we can't play the game above of taking pointwise maximum values, as we lose the martingale property. The maximum of two martingales that distribute their capital in different regions will require perhaps double the initial capital.}
\andre{looks fine to me. Perhaps call it a sketch of proof? "following the rest of the proof" could mean a number of things}
\noam{But in this case, it is really just following the rest of the proof, verbatim.}

\begin{proof}
We give a variant of the proof of Proposition~\ref{prop:C-nonuniform:restatement}. Fix $k\in \omega$. Let~$\Gamma$ be a functional. We may assume that the initial capital of the optimal supermartingale is bounded by~1. Let~$P$ be the collection of all martingales with initial capital $\le k$. We can code~$P$ as a computably bounded $\Pi^0_1$ class in Baire space as follows: for each~$n$, the value $f_M(n)$ of the function coding~$M$ codes, for each binary string~$\s$ of length at most~$n$, one of the dyadic closed intervals $[k/2^n, (k+1)/2^n]$ (for integer $k$ between~0 and $k2^{n+|\s|}$) containing the value of~$M(\s)$. Let~$P_n$ be the collection of strings of length~$n$ coding initial segments of martingales in this way. 

Let~$\Gamma$ be a $k$-valued functional, and suppose that $\Gamma(M)$ is total for all $M\in P$. As above the recursion theorem gives us some~$e$ for which we can define $J(e)$. By effective compactness, there is some~$n$ such that $\Gamma(\tau,e)\converge$ for every string $\s\in P_n$. Let $c\ge \log_2 k$; we let $\widehat{E}$ be the collection of $\s\in P_{n+c}$ which code a martingale with initial capital $\le 1$; we know that some $\s\in \widehat{E}$ is extendible on~$P$. For $Q\in [\widehat{E}]^k$ we let $\tau_Q\in P_n$ be the (string coding) the sum of the martinagles (with codes) in~$Q$. Again, if some $\s\in Q$ is extendible in~$P$, then so is~$\tau_Q$. The proof then follows the rest of the proof of Proposition~\ref{prop:C-nonuniform:restatement}.
\end{proof}

We remark that the proof of Proposition~\ref{prop:majorizing-nonuniform} does not give the assumption of Lemma~\ref{lem:general_prop_on_finitely_many_functionals}, as adding martingales implies adding their initial capital. We thus ask the following questions for any $k\ge 2$:
\begin{itemize}
	\item Is there a finite collection of functionals~$\Gamma_i$ such that for every martingale~$M$ majorizing~$m$, $\Gamma_i(M) \in \DNC_k$ for some~$i$?
	\item Is there a uniform way to compute a $\DNC_k$ function from a martingale~$M$ majorizing~$m$ whose initial capital is bounded by~1?
\end{itemize}
\andre{perhaps state that we leave this open, rather than making it an open question? I would be surprised it someone works on this} \noam{Not sure what's the difference?}

%%%%%%%%
\subsection{Jordan decomposition on the rationals}
\label{subsec:Jordan}
%%%%%%%%

Given a computable function $f\colon[0,1]\to\R$ of bounded variation, we want to find nondecreasing functions $g,h\colon I_\Q\to\R$ such that $f\uh I_\Q = g-h$, where $I_\Q = [0,1]\cap\Q$. Brattka et al.~\cite{BMN:16} observed that this can be done with a PA degree. Our goal below is to show that finding a Jordan decomposition of $f$ on the rationals is equivalent to finding a martingale that majorizes a related atomless c.e.\ martingale.

\noam{I thought it is better to use Weihrauch. It makes things much clearer to me. It was also implicitly used in what was Observation 2.5. It makes things more precise. For example, I'm not exactly sure what ``compute a DNCk function from a uniform solution to the Jordan decomposition problem on IQ''  means.}

% \andre{If so we should say how it relates to Medvedev reducibility, to keep it coherent with previous results such as Jockusch's on $DNC_k$. The advantage of Medvedev was that it doesn't need names, but I'm happy if we look at this from  a new viewpoint}

% \noam{I tried this but couldn't see how to do this in a reasonable way. The point is that compression functions do have Turing degree, so makes sense to talk about their Medvedev degree. Martingales and real-valued functions do not have Turing degree, so if you want to discuss Medvedev degrees you *have* to use names. This is not so satisfying. This is why I kept the statement of prop:majorizing-nonuniform as is, rather than talking about the Medvedev degree of names of martingales majorizing the optimal c.e. supermartingale. I added a sentence before its proof. For Jordan decomposition it is uglier, because we didn't build a universal computable function of bounded variation on the unit interval (not sure what that would mean); so now you have infinitely many Medvedev degrees: for each such function~$f$, the degree of names of Jordan decompositions of~$f$ on the rationals. Many of them will not by Muchnik above PA. Weihrauch makes it much cleaner -- a single problem, with the names already built in in this setting.}

% \andre{OK, thanks for the explanation. Indeed if there is no universal object we don't have a mass problem, but still Wrauch problem. Weihrauch also fits nicely with the Christmas time.}

A natural formalisation of this equivalence uses \emph{Weihrauch reducibility}. The objects compared by this reducibility are binary relations, which can be thought of as pairs of ``instances'' and ``solutions''. For instance, in this section we consider the problem of finding the positive part of a Jordan decomposition on the rationals:
% \begin{itemize}
% 	\item $\PJD{\Q}$ is the problem whose instances are continuous functions~$f$ on $[0,1]$ of bounded variation, and solutions are non-decreasing functions $g\colon I_{\Q}\to \R$ for which $g-f\rest{I_\Q}$ is also non-decreasing. 
% \end{itemize}
\begin{itemize}
	\item $\JD{\Q}$ is the problem whose instances are continuous functions~$f$ on $[0,1]$ of bounded variation, and solutions are Jordan decompositions $(g,h)$ of $f\rest{I_\Q}$.
\end{itemize}
If~$A$ and~$B$ are Weihrauch problems, then we say that $A$ is \emph{Weihrauch reducible} to~$B$ (and write $A\le_{W} B$) if there are two computable mappings $\psi_{\inst}$ and $\psi_{\sol}$ satisfying: for every name~$a$ for an instance for~$A$, $\psi_{\inst}(a)$ is a name for an instance for~$B$, such that whenever~$c$ is a name for a $B$-solution for the instance named by $\psi_{\inst}(a)$, $\psi_{\sol}(a,c)$ is a name for an $A$-solution for the instance named by~$a$. If $\psi_{\sol}$ does not make use of~$a$ then the reduction is called \emph{strong}. Note that the functions $\psi_\inst$ and $\psi_\sol$ are not required to induce functions on the instances and solutions themselves; two names of the same $A$-instance may be mapped by $\psi_\inst$ to names of distinct $B$-instances, and the same holds for the solutions. When we define reductions, though, in order to make things readable, we blur the distinction between names and the objects they name. 

\smallskip

To show the PA-completeness of the Jordan decomposition problem, we will prove the equivalence of the problem $\JD{\Q}$ with a martingale domination Weihrauch problem. Let us define a \emph{lower semicontinuous presentation} of a martingale~$M$ to be a sequence $\seq{M_s}$ of rational-valued martingales such that $M_s \le M_{s+1}$ and $M = \lim_s M_s$. If $\seq{M_s}$ is computable then we also call it a \emph{c.e.\ presentation} of~$M$. We define the following Weihrauch problem.
\begin{itemize}
	\item $\AMD$ is the problem whose instances are lower semicontinuous presentations $\seq{M_s}$ of atomless martingales~$M$; $\AMD$-solutions for~$\seq{M_s}$ are martingales majorizing~$M$ (not necessarily atomless).  
\end{itemize}
We will show:

\begin{prop} \label{prop:sW_equivalence_of_MD_and_PJDQ}
	The  problems $\JD{\Q}$ and $\AMD$ are Weihrauch equivalent. 
\end{prop}

From this we can deduce the following: 

\begin{cor} \label{cor:PA-completeness_of_JD_on_rationals}
	Suppose that~$X$ is an oracle such that for every computable function $f\colon [0,1]\to \R$ of bounded variation, $X$ can compute a Jordan decomposition $(g,h)$ of $f\rest{I_{\Q}}$. Then $X$ is PA-complete. In fact, there is a single computable function $f\colon [0,1]\to \R$ of bounded variation such that any~$X$ computing a Jordan decomposition of $f\rest{I_{\Q}}$ is PA-complete.
\end{cor}

\begin{proof}
	Let $(\psi_\inst,\psi_\sol)$ be a Weihrauch reduction of $\AMD$ to~$\JD{\Q}$. Let $\seq{M_s}$ be a c.e.\ presentation of the atomless martingale~$M$ given by Proposition~\ref{prop:mdom}. Let $f = \psi_\inst(\seq{M_s})$. Since~$\psi_\inst$ is a computable mapping and~$\seq{M_s}$ is computable, so is $f$. Suppose that~$X$ computes a Jordan decomposition $(g,h)$ of $f\rest{I_\Q}$. Then $N=\psi_\sol(\seq{M_s}, (g,h))$ is also $X$-computable; since~$N$ majorizes~$M$, $X$ is PA-complete. 
\end{proof}

On the other hand, Proposition~\ref{prop:sW_equivalence_of_MD_and_PJDQ} also allows us to transfer our non-uniformity result. 

\begin{cor} \label{cor:non-uniformity-of_JD}
	For any computable function $f\colon [0,1]\to \R$ of bounded variation and any~$k$, there is no uniform way of computing a $\DNC_k$ function from (a name of) any Jordan decomposition $(g,h)$ of $f\rest{I_\Q}$. 
\end{cor}

\begin{proof}
	Suppose that~$\Gamma$ is a Turing functional mapping (names of) pairs of real-valued functions~$(g,h)$ on~$I_\Q$ to $k$-valued functions on~$\omega$. Let $f\colon [0,1]\to \R$ be computable of bounded variation; we need to show that there is some Jordan decomposition $(g,h)$ of~$f$ such that $\Gamma(g,h)\notin \DNC_k$. 

	Let $(\vphi_\inst,\vphi_\sol)$ be a Weihrauch reduction of~$\JD{\Q}$ to~$\AMD$; let $\seq{M_s} = \vphi_\inst(f)$, and let $M = \lim_s M_s$. Note that $\seq{M_s}$ is computable, so~$M$ is c.e. Recalling that~$m$ is the optimal c.e.\ supermartingale, fix some $d>0$ such that $dm \ge M$. 

	We define a functional~$\Theta$ by letting $\Theta(N) = \Gamma(\psi_\sol(f,dN))$. By Proposition~\ref{prop:majorizing-nonuniform}, there is a martingale~$N$ majorizing~$m$ such that $\Theta(N)\notin \DNC_k$. Then $(g,h) = \psi_\sol(f,dN)$ is a Jordan decomposition of $f\rest{I_\Q}$ such that $\Gamma(g,h)\notin \DNC_k$. 
 \end{proof}

 \begin{remark} 
 	The proof of Corollary~\ref{cor:non-uniformity-of_JD} shows that we can compute a Jordan decomposition on~$I_\Q$ of a computable function~$f$ of bounded variation, uniformly given a martingale~$N$ majorizing~$m$ and a computable index for (a name of)~$f$. This is because the constant~$d$ can be computed given a computable index for~$f$. 
 \end{remark}

It remains to prove Proposition~\ref{prop:sW_equivalence_of_MD_and_PJDQ}. 

\medskip

The first step is to translate the problem to the dyadic rationals, $\Q_2$. Let $I_{\Q_2} = [0,1]\cap \Q_2$. 
We define the following Weihrauch problem:
\begin{itemize}
	\item $\JD{\Q_2}$: instances are continuous functions $f\colon [0,1]\to \R$ of bounded variation; solutions for~$f$ are Jordan decompositions of $f\rest{I_{\Q_2}}$.
\end{itemize}

\begin{lem} \label{lem:transfer_to_dyadic}
	The problems $\JD{\Q}$ and $\JD{\Q_2}$ are strong Weihrauch equivalent. 
\end{lem}

\begin{proof}
	Let $b\colon [0,1]\to [0,1]$ be a computable, order-preserving bijection such that $b[I_{\Q}] = I_{\Q_2}$; we get this by extending a computable, order preserving bijection between $I_{\Q}$ and $I_{\Q_2}$. Note that $b^{-1}$ is also computable. 

	To reduce $\JD{\Q}$ to $\JD{\Q_2}$, map an instance~$f$ to $f\circ b$; note that if~$f$ has bounded variation, then so does $f\circ b$, in fact $V_{f\circ b}(1) = V_f(1)$. On the solution side, map a pair $(g,h)$ of functions defined on $I_{\Q_2}$ to the pair $(g\circ b^{-1}, h\circ b^{-1})$. 

	To reduce $\JD{\Q_2}$ to $\JD{\Q}$, map an instance~$f$ to itself; On the solution side, map a pair $(g,h)$ of functions defined on $I_\Q$ to the pair $(g\rest{I_{\Q_2}}, h\rest{I_{\Q_2}})$.  
\end{proof}

Just for notational simplicity later, define the following Weihrauch problem:
\begin{itemize}
	\item $\PJD{\Q_2}$: instances are continuous functions $f\colon [0,1]\to \R$ of bounded variation; solutions for~$f$ are functions $g\colon I_{\Q_2}\to \R$ such that $(g,g-f\rest{I_{\Q_2}})$ is a Jordan decomposition of $f\rest{I_{\Q_2}}$. 
\end{itemize}
It is clear that $\PJD{\Q_2}$ is Weihrauch equivalent to $\JD{\Q_2}$; The reduction of $\JD{\Q_2}$ to~$\PJD{\Q_2}$ is not strong. 

\smallskip

The variation $V_f$ of a function $f\colon I_{\Q_2}\to \R$ is defined as usual, except that the partitions have binary rationals as endpoints. If $f\colon [0,1]\to \R$ is continuous of bounded variation then so is $\bar f = f\rest{I_{\Q_2}}$, and $V_{\bar f} = V_f\rest{I_{\Q_2}}$. 

We will transform functions on $I_{\Q_2}$ of bounded variation into signed measures on $[0,1)$. To do this, we associate binary strings with dyadic rational numbers -- the endpoints of the associated intervals -- in the natural way. For the empty string~$\lambda$ we let $l_\lambda = 0$ and $r_\lambda = 1$; for any finite binary string~$\s$, we let $l_{\sigma\conc 0} = l_\sigma$, $r_{\sigma\conc 1} = r_\sigma$, and $r_{\sigma\conc 0} = l_{\sigma\conc 1} = \nicefrac{(l_\sigma+r_\sigma)}{2}$. We write $[\s)$ for the half-open interval $[l_\sigma,r_\sigma)$. 

For a function $f\colon I_{\Q_2}\to \R$ of bounded variation, there is a (unique) signed measure $\mu^f$ on $[0,1)$ defined by 
\[
\mu([\s)) = 	 f(r_\s) - f(l_\s).
\]
The map $f\mapsto \mu^f$ is computable. Observe that $f$ (defined on $I_{\Q_2}$) is non-decreasing if and only if $\mu^f$ is non-negative. If $f\colon [0,1]\to \R$ is continuous then we write $\mu^f$ for $\mu^{f\rest{I_{\Q_2}}}$. For clarity, for a signed measure~$\mu$ and $\s\in 2^{<\w}$, we write $\mu(\s)$ for $\mu([\s))$. 
\begin{obs} \label{obs:cirterion_for_PJD_solution}
	Let $f\colon [0,1]\to \R$ be continuous of bounded variation. A function $g\colon I_{\Q_2}\to \R$ is a $\PJD{\Q_2}$-solution for~$f$ if and only if $\mu^g \ge 0$ and $\mu^{g}\ge \mu^f$. This is because $\mu^{g-f\rest{I_{\Q_2}}}= \mu^g -\mu^{f\rest{I_{\Q_2}}} = \mu^g-\mu^{f}$.
\end{obs}
The operation $f\mapsto \mu^f$ has an inverse of sorts: for any (finite) signed measure~$\mu$ on $[0,1)$ we define $f_\mu\colon [0,1]\to \R$ by letting 
\[
	f_\mu(x) = \mu([0,x)). 
\]
This is known as the \emph{cummulative distribution function} of~$\mu$. The function $f_\mu$ is not necessarily $\mu$-computable (rather it is $\mu$-left-c.e.); however $f_\mu\rest{I_{\Q_2}}$ is $\mu$-computable (uniformly), because for $q\in I_{\Q_2}$ positive we have 
\[
	f_\mu(q) = \sum_{\tau \le \s, |\tau|= |\s|} \mu(\tau)
\]
for any $\s$ such that $q = r_\s$. If $g\colon [0,1]\to \R$ is continuous of bounded variation then $f_{\mu^g} = g-g(0)$. A measure~$\mu$ is atomless if and only if~$f_\mu$ is continuous. 

\smallskip

The Hahn decomposition of a signed measure~$\mu$ produces the \emph{variation} $V_\mu$ of~$\mu$ (often denoted by $|\mu|$); it is the least measure~$\nu$ satisfying $\nu(A)\ge |\mu(A)|$ for all Borel~$A$. The measure $V_\mu$ is $\mu$-left c.e., uniformly: there is a computable mapping taking~$\mu$ to a lower semicontinuous presentation of~$V_\mu$. This is because for all~$\s$,
\[
		V_\mu(\s) = \sup_{k\ge |\s|} \sum \left\{ |\mu(\tau)| \,:\,  \tau\succeq \s\andd |\tau|=k   \right\}.
\]
For a continuous~$f$ on~$[0,1]$ of bounded variation we have $V_{\mu^f} = \mu^{V_f}$. 

\smallskip

Finally, we replace martingales by measures in the familiar way: a martingale~$M$ corresponds to the measure defined by $\mu(\s) = 2^{-|\s|} M(\s)$. We thus assume that instances and solutions of $\AMD$ are measures rather than martingales. We are ready to prove one direction of Proposition~\ref{prop:sW_equivalence_of_MD_and_PJDQ}:

\begin{prop} \label{prop:reducing_JD_to_AMD}
	$\PJD{\Q_2}$ is strong Weihrauch reducible to $\AMD$. 
\end{prop}

\begin{proof}
	On the instance side, we map a continuous function $f\colon [0,1]\to \R$ of bounded variation to a lower semicontinuous presentation of $V_{\mu^f} = \mu^{V_{f}}$; we observed that this can be done computably. Since~$f$ is continuous, so is $V_f$, so $\mu^{V_f}$ is atomless.

	On the solution side, map a measure $\nu$ to $f_\nu\rest{I_{\Q_2}}$.

	To show this works, suppose that~$\nu \ge \mu^{V_{f}}$; then $\nu \ge \mu^f$, as $\mu^{V_{f}}\ge \mu^f$. Also $\mu^g = \nu \ge 0$. By Observation~\ref{obs:cirterion_for_PJD_solution}, $g$ is a $\PJD{\Q_2}$-solution for~$f$. 
\end{proof}

In the other direction, we need two facts.

\begin{lem} \label{lem:Jordan:decomposition_majorizes_variation_measure}
	 Let $(g,h)$ be a Jordan decomposition of a function $f\colon I_{\Q_2}\to \R$ of bounded variation. Then $\mu^g+\mu^h \ge \mu^{V_f}$. 
\end{lem}
\begin{proof}
	The minimality property of $V_{\mu^f}$ means that it suffices to show that for all~$\s$, $\max \{ \mu^g(\s),\mu^h(\s)\}\ge |\mu^f(\s)|$. If $\mu^f(\s)\ge 0$ then $\mu^g(\s)\ge \mu^f(\s)$ as $\mu^g\ge \mu^f$ (by Observation~\ref{obs:cirterion_for_PJD_solution}). If $\mu^f(\s)<0$ then $\mu^{h}(\s) = \mu^g(\s)-\mu^f(\s) = \mu^g(\s)+ |\mu^f(\s)|\ge |\mu^f(\s)|$ because $\mu^g\ge 0$.
\end{proof}

The main technical fact  is  taken from the proof of   Theorem~3.5 of~\cite{Freer.Kjos.ea:14} by Freer et al.\ (joint with Rute). That theorem    states that any continuous non-decreasing interval-c.e.\ function $f\colon [0,1] \to \mathbb R$ is of the form $V(g,[0,x])$ for some computable function $g$. (To say that $f$ is interval-c.e.\ means that  the real $f(y)-f(x)$ is left-c.e., uniformly in rationals $x<y$.) This implies   that the associated ``slope" martingale $M(\sigma) = (f(r_\sigma)- f(l_\sigma))/(r_\sigma- r_\sigma)$ is left-c.e. Here we use the equivalent notation of   measures, rather than of martingales.
\andre{if it stays like this we should say that we paraphrase the result}
\noam{Sure. I can add a footnote explaining how this is a paraphrasing.}
\andre{again, I think better in the main text. more and more footnotes makes it look patched}
\andre{please leave the statement of the theorem of [10], for context; and good to know that thm was also about effective analysis.}
\noam{Sorry, I don't understand.}
\begin{prop}[\cite{Freer.Kjos.ea:14}] \label{prop:Freer_et_al}
	There is a computable mapping taking any lower semicontinuous presentation $\seq{\nu_s}$ of an atomless measure~$\nu$ to a continuous function $g\colon [0,1]\to \R$ of bounded variation such that $\nu = V_{\mu^g}$. 
\end{prop}

\begin{proof}[Sketch of proof]
	We define a signed measure~$\eta$ and let $g=f_\eta$ (so $\eta = \mu^g$). The rough idea is as follows. By stage~$s$ we have defined~$\eta(\s)$ for all $\s$ of length $\le \ell_s$ for some $\ell_s\in \omega$, with $|\eta(\s)|\le \nu_s(\s)$ for all such~$\s$. At stage~$s$ we define~$\eta(\tau)$ for longer strings~$\tau$ (preserving $|\eta(\tau)|\le \nu_s(\tau)$), by letting $|\eta(\tau\conc i)| = \nu_s(\tau\conc i)$ for the~$i$ for which the latter is the smaller between $\nu_s(\tau\conc 0)$ and $\nu_s(\tau\conc 1)$; but we keep the sign of $\eta(\tau\conc i)$ the same as that of $\eta(\tau)$. As we go along, at every level $n\ge \ell_s$, at most $2^{\ell_s}$ many strings~$\tau$ of length~$n$ have $|\eta(\tau)|\ne \nu_s(\tau)$. As~$\nu_s$ is atomless (because~$\nu$ is), eventually the discrepancy between~$\nu_s(\s)$ and $\sum_{\tau\succ \s\andd |\tau|=n} |\eta(\tau)|$ is small for each~$\s$ of length~$\ell_s$, which is when we halt stage~$s$ and declare the next value $\ell_{s+1}$. To get $g=f_\eta$ computable from~$\seq{\nu_s}$, we need to ensure that~$\ell_s$ is sufficiently long so that $\nu_s(\s)\le 2^{-s}$ for all~$\s$ of length~$\ell_s$, which again is possible because~$\nu_s$ is atomless.\footnote{The proof as written in~\cite{Freer.Kjos.ea:14} uses martingales instead of measures; to translate to the notation of that paper, $\nu(\s) = 2^{-|\s|}M(\s)$ and $\eta(\s)= 2^{-|\s|}L(\s)$. Lemma 3.3 of~\cite{Freer.Kjos.ea:14} constructs a computable signed measure~$\eta$ such that $\nu = V_\eta$; In Theorem 3.5, the construction is modifed to get $\eta = \mu^g$ with~$g$ computable, starting with a function~$f$ such that $\nu = \mu^f$.}
\end{proof}
\andre{I'm not sure this is better than the previous version. Here you can't rely on the proof in the Freer paper because the notation is too different.  In my version I relied on the  $V_L$ defined precisely in [10]. Are you paraphrasing that argument but avoiding martingales?  Maybe you can rewrite this version so that the proof in [10] gives the detail.}
\noam{The proof sketched here is essentially identical to the proof in the paper. Martingales and measures are really the same thing. All you do is divide by $2^{|\s|}$. $V_\mu$ is precisely the same as $V_L$, just in the language of measures.}
\andre{ok pls insert the content of this comment into the main text}
\noam{This is the content of footnote 5, which I really think should remain a footnote.}

The following now completes the proof of Proposition~\ref{prop:sW_equivalence_of_MD_and_PJDQ}:

\begin{prop} \label{prop:reducing_AMD_to_JD}
	$\AMD$ is strong Weihrauch reducible to $\JD{\Q_2}$. 
\end{prop}

\begin{proof}
	On the instance side, using Proposition~\ref{prop:Freer_et_al}, map a lower semicontinuous presentation $\seq{\nu_s}$ of an atomless measure function~$\nu$ to some continuous~$f$ such that $\nu = V_{\mu^f} = \mu^{V_f}$. On the solution side, map $(g,h)\colon I_{\Q_2}\to \R$ to $\mu^g+\mu^h$. Lemma~\ref{lem:Jordan:decomposition_majorizes_variation_measure} says that this works. 
\end{proof}

%%%%%%%%
%%%%%%%%
\section{A \texorpdfstring{$K$}{K}-compression function without PA degree}
\label{sec:K-compression}
%%%%%%%%
%%%%%%%%

 We provide a    proof that  there is a $K$-compression function that does not have PA degree. This  should be considered a warm-up for the somewhat more involved proof of Theorem~\ref{thm:main-separation}, which by Propositions~\ref{prop:equiv} and~\ref{prop:continuous-to-discete} implies the present result.

As this is a warm-up, we introduce notation which may appear cumbersome at present, but will be useful later. In the current argument, we work in the space $\id^\w$, which, recall, is the space of identity-bounded functions. We also let $\id^{\le \w} = \id^\w\cup \id^{<\w}$ be the collection of idnetity-bounded sequences, finite and infinite. For $\s\in \id^{<\w}$, we let 
\[
	[\s] = \left\{ f\in \id^{\le \w} \,:\,  \s\preceq f \right\}. 
\]
The sets $[\s]\cap \id^\w$ are the basic clopen subsets of $\id^\w$, and generate the topology on that space, which is the topology inherited from Baire space. 

For convenience, we treat prefix-free complexity~$K$ as a function on $\omega$ (via the length-lexicographical ordering of binary strings). The weight of a function  $f\in \w^{\le \w}$~is
\[
\wt(f) = \sum_{n\in \dom f} 2^{-f(n)}.
\]
We say that $f\in \w^\w$ has \emph{finite weight} if $\wt(f)<\infty$. For a set $A\subseteq \id^{\le \w}$ and real number $r$, we let
\[
	A_{\le r} = \left\{ f\in A  \,:\,  \wt(f)\le r   \right\},
\]
and we similarly define $A_{<r}$ and $A_{>r}$. If $P\subseteq \idb$ is a $\Pi^0_1$ class, then for any rational number~$q$, $P_{\le q}$ is a $\Pi^0_1$ class as well. This is not usually true for $P_{<q}$ (let alone $P_{>q}$). Note that the space $\idb$, by definition, is computably bounded, and so $\Pi^0_1$ subclasses of~$\idb$ are effectively compact: from a cover of such a set generated by a c.e.\ collection of basic clopen sets, we can effectively find a finite sub-cover. Also, every PA degree computes an element of each nonempty such set. 

% If $P\subseteq\BS$ and $q\in\Q$, then we let $P_{\le q} = \{f\in P\colon \wt(f)\leq q\}$. Define $P_{<q}$ similarly. Note that if $P$ is a (computably bounded) $\Pi^0_1$ class, then $P_{\le q}$ is too. (This is not usually true of $P_{<q}$.)

Let $P^K = \{f\in\idb\colon f\leq K\}$. Note that $P^K$ is a $\Pi^0_1$ class. 

\begin{lem} \label{lem:PK_has_finite_weight}
	$P^K$ contains a finite weight function; indeed, $P^K_{\le 3}\ne \emptyset$.
\end{lem}

\begin{proof}
	This follows from the fact that $\wt(K) = \sum_{n\in\omega} 2^{-K(n)}<1$. Let $K^*(n) = \min\{K(n), n\}$. Then $K^*\in P^K$ and $\wt(K^*)\leq \wt(K)+\sum_{n\in\omega} 2^{-n} < 3$.
\end{proof}

% We claim that $P^K$ contains a finite  weight function; this follows easily from the fact that $\wt(K) = \sum_{n\in\omega} 2^{-K(n)}<1$. In particular, let $K^*(n) = \min\{K(n), n\}$. Then $K^*\in P^K$ and $\wt(K^*)\leq \wt(K)+\sum_{n\in\omega} 2^{-n} < 3$. Therefore, $P^K_{\leq 3}$ is a computably bounded and nonempty $\Pi^0_1$ class containing only bounded weight $K$-bounded functions. 
This proves that every PA degree computes a $K$-bounded function of finite weight (a fact we already saw in the introduction). Note that as $\wt(K)<1$, a $K$-bounded function of finite weight can be, by finite alteration, changed to a $K$-bounded function with weight bounded by~$1$, so such functions have the same Turing degrees as $K$-compression functions.

Our goal is to prove that being of  PA degree is \emph{not necessary} to compute a $K$-bounded function of finite weight.

\begin{thm}\label{thm:K-bounded}
There is a $K$-bounded function $f\colon\omega\to\omega$ of finite weight that does not have PA degree.
\end{thm}

\begin{proof}
We build $f$ using a forcing argument. The \emph{forcing conditions} are triples of the form $(\s,P,q)$ where:
\begin{itemize}
	\item $\s\in \id^{<\w}$;
	\item $P\subseteq P^K \cap [\s]$ is a $\Pi^0_1$ class such that:
	\begin{itemize}
		% \item every $h\in P$ extends $\s$;
		\item if $h\in P$, $g\le h$, and $g\in [\s]$, then $g\in P$;
	\end{itemize}
	\item $q\in\Q$ and $P_{\le q}\neq\emptyset$.
\end{itemize}
The condition $(\s,P,q)$ should be thought of as saying that $f\in P_{\le q}$. We say that $(\tau,R,s)$ \emph{extends} $(\s,P,q)$ if $\s\preceq\tau$, $R\subseteq P$, and $s\leq q$. Note that $(\seq{},P^K,3)$ is a condition, so the set of conditions is nonempty. 

For a filter $G$ of forcing conditions, we let 
% \[ f_G = \bigcup \s\Cyl{(\s,P,q)\in G\text{ for some $P$ and $q$}} .\]
\[
f_G = \bigcup\; \set{\s}{(\s,P,q)\in G\text{ for some $P$ and $q$}}.
\]
Then $f_G\in \id^{\le \w}$. If $(\s,P,q)$ is a condition, then we can find~$\tau$ properly extending $\s$ such that $(\tau,P\cap[\tau],q)$ %\noam{$P\cap \tau$?}
is also a condition (take $\tau$ to be an initial segment of a function witnessing that $P_{\le q}$ is nonempty). This shows that if $G$ is only mildly generic, then $f_G$ is defined on all of $\w$. 

\begin{lem}\label{lem:silly_stuff}
	Suppose that $(\s,P,q)\in G$. Then $f_G\in P_{\le q}$. 
\end{lem}
\begin{proof}
	Let $\tau\prec f_G$. Then there is a condition $(\tau,Q,s)\in G$. By extending this condition (and possibly $\tau$), we may assume that $(\tau,Q,s)$ extends the condition $(\s,P,q)$. Since $Q_{\le s}$ is nonempty and $Q_{\le s} \subseteq [\tau]\cap P_{\le q}$, we have that $[\tau]\cap P_{\le q}$ is nonempty. This is true for all $\tau\prec f_G$. Since $P_{\le q}$ is closed, we have $f_G\in P_{\le q}$. 
\end{proof}

By definition, $P\subseteq P^K$ for any condition $(\s,P,q)$, so $f_G$ is $K$-bounded. Lemma~\ref{lem:silly_stuff} also implies that $\wt(f_G)$ is finite.

\medskip
There is not much difference between $P_{\le q}$ and $P_{<q}$. 

\begin{lem}\label{lem:lose_weight_now_ask_me_how}
Let $(\s,P,q)$ be a condition. Then $P_{<q}$ is nonempty. 
\end{lem}
\begin{proof}
Suppose not. Then $P_{\le q}$ is nonempty and every element of $P_{\le q}$ has weight exactly $q$, i.e., $P_{\le q} = P_{=q}$. This gives us an algorithm for computing $\emptyset'$. Note that if $m$ enters $\emptyset'$ at stage $s$, then $K(s)\le^+ m$. Hence it suffices, given any $m<\w$ to find some $n<\w$ such that $K(x)\ge m$ for all $x\ge n$.

To do so, let $T$ be a computable subtree of $\id^{<\w}$ such that $[T] = P_{=q}$. For $r<q$, recall that $\id^{<\w}_{>r}$ is the collection of finite $\tau$ such that $\wt(\tau)>r$.\footnote{Note that it is possible that $\wt(\tau)< r$ but every infinite extension of~$\tau$ in $\idb$ has weight $>r$; indeed, $[\tau] \subseteq \idb_{>r}$ if and only if $\wt(\tau)+2^{-|\tau|+1}>r$. Nonetheless, $\idb_{>r}$ is the open set generated by $\id^{<\w}_{>r}$.}  
%be the set of finite strings $\s\in \id^{<\w}$ with $\wt(\s)>r$; so $\id^\w_{>r}$ is the open subset of $\id^\w$ generated by $O_r$. (Note however that it is possible that $[\s]\cap \id^\w\subseteq \id^\w_{>r}$ for strings $\s\notin O_r$; $[\s]\cap \id^\w\subseteq \id^\w_{>r}$ if and only if $\wt(\s)+2^{-|\s|+1}> r$.) 
Let $T_n$ be the set of strings on $T$ of length $n$. Since $P_{=q}\subseteq \id^\w_{>r}$ and $\id^\w$ is compact, for every $r<q$ there is an $n\in\w$ such that $T_n\subseteq \id^{<\w}_{>r}$; such~$n$ can be of course found effectively from~$r$. If $T_n \subseteq \id^{<\w}_{>q-2^{-m}}$, then $K(x)>m$ for all $x\ge n$. For we know that there is some $\s\in T_n$ which is extendible ($[\s]\cap P_{=q}\ne \emptyset$); if $h\in [\s]\cap P_{=q}$, $x\ge |\s|$, and $\wt(h)-\wt(\s)< 2^{-m}$ then $h(x)>m$. Since $(\s,P,q)$ is a condition, we know that $h \le K$. 
\end{proof}

\begin{remark}
By the foregoing fact, if $(\rho,R,t)$ is a condition then there is some $t'<t$ such that $(\rho,R,t')$ is a condition as well. Thus, by genericity, if $(\s,P,q)\in G$, then there is some $q'<q$ such that $(\s,P,q')\in G$. By Lemma~\ref{lem:silly_stuff}, $f_G\in P_{\le q'}$, and so $f_G\in P_{<q}$.	
\end{remark}

The main work is  to show that $f_G$ does not have PA degree. This will follow from genericity (and Lemma~\ref{lem:silly_stuff}), once we show that for any Turing functional $\Gamma$, the collection of conditions 
\[
D_\Gamma = \left\{ (\s,P,q)\,:\, (\forall h\in P_{\le q})\; \Gamma(h)\notin \DNC_2 \right\}
\]
is dense in our forcing partial order. %below the base condition $(\seq{},P^K,3)$ 

% Let $(\s,P,q)$ be a condition. We want to find an extension of this condition in $E_\Gamma$. 
First, we extend to a condition that gives us some ``breathing room''. We let 
\[
F = \left\{ (\s,P,q)\,:\,  P_{<\wt(\s)+\epsilon}\ne\emptyset \text{ where } \epsilon = (q-\wt(\s))/3 \right\}.
\]

\begin{lem} \label{lem:compression_functions:breathing_room}
	The collection~$F$ of conditions
	%  $(\s,P,q)$ for which there is some~$\epsilon>0$ such that:
	% \begin{itemize}
	% 	\item $\wt(\s)+3\epsilon <q$; and
	% 	\item $P_{<\wt(\s)+\epsilon}$ is nonempty
	% \end{itemize}
	is dense. 
\end{lem}

\begin{proof}
	Let $(\tau,Q,p)$ be a condition. By Lemma~\ref{lem:lose_weight_now_ask_me_how}, let $h\in Q_{<p}$. Pick $\epsilon$ small enough so that $\wt(h)+3\epsilon < p$. Take $\s\prec h$  extending $\tau$ such that $\wt(h)-\wt(\sigma)<\epsilon$. Then $(\s,Q\cap [\s],\wt(\s)+3\epsilon)$ is an extension of $(\tau,Q,p)$ in~$F$. 
\end{proof}

It thus suffices to show that every condition in~$F$ has an extension if $D_\Gamma$. 

Note that if $(\s,P,q)\in F$ with $\epsilon = (q-\wt(\s))/3$, then $(\s,P,\wt(\s)+\epsilon)$ is also a condition; however we will find an extension of $(\s,P,q)$ in~$E_\Gamma$, rather than of $(\s,P,\wt(\s)+\epsilon)$.  \andre{this remark not useful}
\noam{I disagree.}

\medskip

%  we find an extension $(\s^*,P^*,q)$ of $(\s,P,q)$ and a rational $\epsilon>0$ such that, letting $r = \wt(\s^*)$, we have:
% \begin{itemize}
% 	\item $r+3\epsilon < q$; and
% 	\item $P^*_{<r+\epsilon}$ is nonempty. 
% \end{itemize}
% Note however that even though this means that $(\s^*,P^*,r+\epsilon)$ is a condition, we will find, in $E_\Gamma$, an extension of $(\s^*,P^*,q)$ rather than of $(\s^*,P^*,r+\epsilon)$. 

% Finding $\s^*$, $P^*$ and $\epsilon$ is easy. By Lemma~\ref{lem:lose_weight_now_ask_me_how}, let $h^*\in P_{<q}$. Pick $\epsilon$ small enough that $\wt(h^*)+3\epsilon < q$. Take $\sigma^*\prec h^*$ such that $\wt(h^*)-\wt(\sigma^*)<\epsilon$ and let $P^* = P\cap [\s^*]$.

Fix some $(\s^*,P^*,q)\in F$; let $r = \wt(\s^*)$ and $\epsilon = (q-r)/3$. 

\smallskip

As we did in the proof of Proposition~\ref{prop:C-nonuniform:restatement}, we define a partial computable process \andre{why `process'? what makes a process partial computable?  this is vague, we actually define a p.r.\ function $\psi(e)$. In the end there is  $p$ such that $J(p(e)) = \psi(e)$ for each $e$, and then we use a FP obtained through the  RT for $p$ to define the extension that makes   $\Gamma(h)$  not DNC2} \noam{This is not as complicated as you put it, and was used above already (proof of Prop 2.5). We define a p.r. function $\psi$; by the recursion theorem, there is an $e$ such that $J(e) = \psi(e)$.  I think that ``process'' is fine, since we only use one input. I added a footnote above, but I'm not sure it's necessary.} \andre{Sorry can't spot that footnote, is it 5? If it's a p.r. function then we should say that, why confuse the reader with a vaguely defined notion of "process"? CTT is certainly ok to use without mention} 
\noam{ It is footnote 2 currently on page 10. I don't think that readers will be confused. This kind of argument has been used often, so will be familiar. This terminology allows us to avoid notation which is not necessary. Not sure what CTT stands for.}
which may either output $0$ or $1$ (or diverge). The output of this process will be $J(e)$ for some $e$, and by the recursion theorem, we may assume we know $e$ in the definition of this process. Consider the $\Pi^0_1$ class $Q$ obtained from $P^*$ by removing not only all the strings $\tau$ of weight below $r+2\epsilon$ for which $\Gamma(\tau,e)\converge$, but also all strings majorizing such strings $\tau$:
% \[
% Q = \{h\in P^*\,\colon\, \neg(\exists\tau)[\;\tau\in [\s^*]_{<r+2\epsilon}\text{, }\tau\leq h\text{, and }\Gamma(\tau,e)\da\;]\}.
% \]
\[
Q = \{h\in P^*\,\colon\, (\forall\tau\le h)\,\,\tau\notin C\},
\]
where 
\[
	C= \set{\tau\in ( \id^{<\w}\cap [\s^*])_{<r+2\epsilon}}{\Gamma(\tau,e)\converge}.
\]
The point is that if $h\in Q$, $g\le h$, and $\s^*\prec g$, then $g\in Q$.\footnote{Note, however, that $Q$ is not quite the same as the class obtained by removing all $h$ for which there is a $g\le h$ in  $P^*_{\le r+2\epsilon}$ such that $\Gamma(g,e)\converge$. The class $Q$ is smaller, since a witness $\tau$ may not be extendible to an $h$-majorized $g$ of weight at most $r+2\epsilon$.}

\medskip

If $Q_{\le r+2\epsilon} \ne \emptyset$, then our partial computable process does not terminate. Suppose now  that $Q_{\le r+2\epsilon} = \emptyset$. This  is eventually effectively recognised, as $Q_{\le r+2\epsilon}$ is a $\Pi^0_1$ class effectively obtained from $e$. We then use the following:
 \begin{lem} \label{lem:compression_functions:the_set_C}
 	If $Q_{\le r+2\epsilon}=\emptyset$ then we can effectively find some $n<\w$ and some set $E\subseteq \id^{=n}$ such that:
 	\begin{enumerate}
	\item Every $\s\in E$ extends $\s^*$ and $\wt(\s)< r+2\epsilon$ (that is, $E\subseteq \id^{=n}\cap [\s^*]_{<r+2\epsilon}$);
	\item For every $\s\in E$ there is some $\tau\le \s$ (in particular $|\tau|\le |\s|$) in~$C$.
	\item If $\s\in E$, $\s'\in [\s^*]_{<r+2\epsilon}$ and $\s'\le \s$ then $\s'\in E$. 
	\item There is some $\s\in E$ such that $[\s]\cap P^*_{<r+\epsilon}\ne \emptyset$. 
\end{enumerate}
\end{lem}

\begin{proof}
	Let $S$ be a computable tree such that $[S]=P^*$; we may assume that if $\s\in S$, $\s'\le \s$ and $\s'\in [\s^*]$ then $\s'\in S$; this is because if~$\s$ is extendible in $P^*$ (meaning $[\s]\cap P^*\ne\emptyset$) then so is~$\s'$. \andre{ language too vague. what does delay mean?} \noam{Offending sentence removed. this is very straightforward.} Let $S_Q = \set{\s\in S}{(\forall\tau\le \s)\,\,\tau\notin C}$. Then $[S_Q]=Q$; since $Q_{\le r+2\epsilon}$ is empty, by compactness, we can find some~$n$ such that every sequence of length~$n$ in~$S_Q$ has weight $>r+2\epsilon$. We then let $E$ be the collection of sequences of length~$n$ in $S_{<r+2\epsilon}$. (4) holds because $P^*_{<r+\epsilon}$ is nonempty; some $\s$ of length~$n$ has an extension in~$P^*$ of weight $<r+\epsilon$; necessarily, $\s\notin S_Q$. 
\end{proof}

% By effective compactness, we can find an $n\in\w$ and a subset $E\subseteq \id^{n}$ such that:
% \begin{enumerate}
% 	\item Every $\s\in E$ extends $\s^*$ and $\wt(\s)< r+2\epsilon$ (that is, $E\subseteq \id^n\cap [\s^*]_{<r+2\epsilon}$);
% 	\item For every $\s\in E$ there is some $\tau\le \s$ (in particular $|\tau|\le |\s|$) in $[\s^*]_{<r+2\epsilon}$
% 	 such that $\Gamma(\tau,e)\converge$. 
% 	\item If $\s\in E$, $\s'\in [\s^*]_{<r+2\epsilon}$ and $\s'\le \s$ then $\s'\in E$. 
% 	\item There is some $\s\in E$ such that $[\s]\cap P^*_{<r+\epsilon}\ne \emptyset$. 
% \end{enumerate}
% Note that (4) holds because $P^*_{<r+\epsilon}\cap Q_{\le r+2\epsilon} = \emptyset$ implies that $P^*_{<r+\epsilon}\cap Q = \emptyset$, but by assumption, $P^*_{<r+\epsilon}$ itself is nonempty.

Having obtained~$E$, we let
\[
	\widehat E = E_{<r+\epsilon} = \set{\s\in E}{\wt(\s)< r+\epsilon}.
\]
Condition (4) says that $\widehat E$ is nonempty, indeed some $\s\in \widehat E$ is extendible in $P^*_{< r+\epsilon}$. Now an important point is that if $\s,\s'\in \widehat E$ then the pointwise minimum $\min{(\s,\s')}$ is in $E$, as both $\s$ and $\s'$ extend $\s^*$ and so $\wt(\s)-\wt(\s^*)<\epsilon$, and similarly for $\s'$. This allows us to show the following. For $i\in \{0,1\}$, let $C_i = \left\{ \tau\in C \,:\,  \Gamma(\tau,e) = i \right\}$; we assume that $\Gamma$ maps into $\{0,1\}$-valued functions, so $C = C_0\cup C_1$. 

\begin{lem} \label{lem:getting_one_colour}
	There is some $i\in \{0,1\}$ such that for every $\s\in \widehat E$ there is some $\tau \le \s$ in~$C_i$. 
\end{lem}

\begin{proof}
	For any pair $\s,\s'$ of strings from $\widehat E$, find some $\tau \le \min(\s,\s')$ in~$C$; let $c(\{\s,\s'\}) = \Gamma(\tau,e)$. By Lemma~\ref{lem:DumbRT}, there is a colour $i\in \{0,1\}$ such that for all $\s\in \widehat E$ there is a $\s'\in \widehat E$ such that $c(\{\s,\s'\}) = i$. (This is easy; if it fails for $0$, then a single $\s'$ witnesses it for $1$.) This colour~$i$ is as required.
\end{proof}

A colour~$i$ as given by Lemma~\ref{lem:getting_one_colour} is the output of the computable process just described.

% For any pair $\s,\s'$ of strings from $\widehat C$, find some $\tau \le \min(\s,\s')$ in~$C$; let $c(\{\s,\s'\}) = \Gamma(\tau,e)$ (which we assume is either $0$ or $1$).

% By Lemma~\ref{lem:DumbRT}, there is a colour $i\in \{0,1\}$ such that for all $\s\in \widehat C$ there is a $\s'\in \widehat C$ such that $c(\{\s,\s'\}) = i$. (This is easy; if it fails for $0$, then a single $\s'$ witnesses it for $1$.) This colour is the output of the computable process just described.

\medskip

We now describe the extension of $(\s^*,P^*,r+3\epsilon)$ in $D_\Gamma$. There are two cases. If $Q_{\le r+2\epsilon}$ is nonempty, then $(\s^*, Q,r+2\epsilon)$ is a condition, and $\Gamma(h,e)\diverge$ for all $h\in Q_{\le r+2\epsilon}$. We assume, then, that $Q_{\le r+2\epsilon}$ is empty. Let~$i$ be the outcome of the computable process described above. Let $\s\in \widehat E$ be extendible in $P^*_{<r+\epsilon}$; fix some $h\in P^*_{<r+\epsilon}$ with $\s\prec h$. Let $\tau \le \s$ in~$C_i$.

Let $R = P^*\cap [\tau]$. We claim that $R_{\le r+3\epsilon}$ is nonempty. For we can let $g = \tau\conc h\rest{[|\tau|,\infty)}$. Note that $g\le h$, so $g\in P^*$. And 
\[
\wt(g) =  %= \wt(\tau) + (\wt(h^*)-\wt(\s)) %%%% we could have |\tau|\le |\s|, but who cares
\wt(\tau) + (\wt(h)-\wt(h\rest{|\tau|}))
 \le \wt(\tau) + (\wt(h) - \wt(\s^*)) < (r+2\epsilon)+ \epsilon .\]
Thus $(\tau,R,r+3\epsilon)$ is a condition extending $(\s^*,P^*,r+3\epsilon)$. Every $h\in R$ extends $\tau$, so $\Gamma(h,e) = i = J(e)$. Therefore, $\Gamma(h)\notin \DNC_2$. 
\end{proof}

%%%%%%%%
%%%%%%%%
\section{The discrete covering property}
\label{sec:discrete}
%%%%%%%%
%%%%%%%%

%
% JOE (4/26/19): This section is done
%

In this section, we show that having the discrete covering property is equivalent to computing a $K$-compression function, and that such oracles compute slow growing DNC functions. Recall that we defined the discrete covering property in terms of sequences of subsets of $\omega$. For the first proof in this section, it is convenient to work with sequences $\bar A = \seq{A_n}$ of subsets of $2^{<\omega}$, which is a clearly a harmless translation. Similarly, in the second proof, we work with sequences of subsets of $\omega^{<\omega}$. 

\begin{prop}\label{prop:equiv}
An oracle $D$ computes a $K$-compression function if and only if it has the discrete covering property.
\end{prop}
\begin{proof}
The equivalence is straightforward. First, assume that $D$ has the discrete covering property. Let $A_n = \set{\sigma}{K(\sigma)\leq n}$, so $\bar A = \seq{A_n}$ is a uniformly c.e.\ sequence such that $\wt(\bar A) < 2$. Thus there is a $D$-computable sequence $\bar B = \seq{B_n}$ of finite weight that covers $\bar A$. Define a $D$-computable function $f$ as follow: let $f(\sigma)$ be the least $n$ such that $\sigma\in B_n$. This ensures that $f(\sigma)\leq K(\sigma)$ and $\wt(f) < \wt(B_n) < \infty$. Some finite alteration of $f$ is a $K$-bounded function with weight bounded by $1$, and an application of the Kraft--Chaitin theorem gives us a $D$-computable $K$-compression function.

For the other direction, assume that $F\colon 2^{<\omega}\to 2^{<\omega}$ is a $K$-compression function computable from $D$. Let $\bar A$ be a uniformly c.e.\ sequence of finite weight. Then there is a $c\in\omega$ such that $\sigma\in A_n$ implies that $K(\sigma)\leq n+c$. Hence $\sigma\in A_n$ implies that $|F(\sigma)|\leq n+c$. Define $B_n = \set{\sigma}{|F(\sigma)|\leq n+c}$, so $\bar B$ is a $D$-uniformly computable sequence that covers $\bar A$. Also,
\[
\wt(\bar B)\leq 2^{c+1}\sum_{\sigma\in 2^{<\omega}} 2^{-|F(\sigma)|}\leq 2^{c+1}.
\]
Therefore, $D$ has the discrete covering property.
\end{proof}
\begin{remark}\label{rmk:universal_instance_of_DCP}
Note that by this proof, the sequence $\bar A= \seq{A_n}$ given by $A_n = \set{\sigma}{K(\sigma)\leq n}$  is universal: if $D$ computes a cover for $\bar A$, it has the discrete covering property.
\end{remark} 
\andre{I think this is useful in view of universality for the CCP later}\noam{ Reference added later.}
Recall that an \emph{order function} is a computable, nondecreasing, unbounded function on $\omega$.

\begin{prop}\label{prop:sdnc}
 Let  $h\colon\omega\to\omega\smallsetminus \{0,1\}$ be any order function. Suppose an oracle   $D$ has  the discrete covering property. Then  $D$ computes an $h$-bounded DNC function. 
\end{prop}
\begin{proof}
  Fix an increasing computable function $g$ such that
\begin{equation}\label{eq:g-is-fast}
\lim_{m\to\infty} \frac{h(g(m))}{2^m} = \infty.
\end{equation}
We build a uniformly c.e.\ sequence $\bar A = \seq{A_n}$ of subsets of $h^{<\omega}$ as follows. If $k$ enters $\emptyset'$ at stage $s$, then find a $\tau\in h^{<\omega}$ of length $g(s)$ such that
\[
(\forall n<g(s))\; J_s(n)\downarrow<h(n) \implies \tau(n) = J_s(n).
\]
In other words, except that it must remain $h$-bounded, $\tau$ is trying to be an extension of $J$  at stage $s$. For each $m\leq s$, put $\tau\uh g(m)$ into $A_{k+m+2}$. We act for each $k$ at most once, so
\[
\wt(\bar A) < \sum_{k\in\omega}\sum_{m\in\omega} 2^{-k-m-2} = \sum_{k\in\omega} 2^{-k-1} = 1.
\]
Let $\bar B$ be a $D$-computable sequence of subsets of $h^{<\omega}$ such that $\wt(\bar B)$ is finite. By removing finitely many elements, we may assume that $\wt(\bar B)\leq 1$.

There are two cases, much like in the proof of Proposition~\ref{prop:mdom}. First, assume that there is a $k$ such that for all $m$, there is a $\tau\in B_{k+m+2}$ of length $g(m)$ such that
\[
(\forall n<g(m))\; J(n)\downarrow<h(n) \implies \tau(n) = J(n).
\]
Define a $D$-computable function $f\colon\omega\to\omega$ such that if $n\in [g(m-1),g(m))$, then $f(n)$ is different from $\tau(n)$ for every $\tau\in B_{k+m+2}$ of length $g(m)$. Our assumption guarantees that if $f(n)<h(n)$, then $f(n)\neq J(n)$. Note that there are at most $2^{k+m+2}$ elements of $B_{k+m+2}$, hence we can ensure that $f(n)\leq 2^{k+m+2}$ for all such $n$. But by \eqref{eq:g-is-fast}, we have $h(g(m-1)) > 2^{k+m+2}$ for all sufficiently large $m$. Therefore, $f(n)<h(n)$ for all sufficiently large $n$. By taking a finite modification of $f$, we get a $D$-computable $h$-bounded DNC function.

If the first case fails, then for every $k$ there is an $m$ and a $t$ such that
\[
(\forall \tau\in B_{k+m+2}\cap \omega^{g(m)})(\exists n<g(m))[\; J_t(n)\downarrow<h(n)\text{ but }\tau(n) \neq J_t(n) \;].
\]
Note that we can find such an $m$ and $t$ effectively from $\bar B$, hence from $D$. Let $s(k) = \max\{m,t\}$, so $s$ is a $D$-computable function. By the construction of $\bar A$, it cannot be the case that $k$ enters $\emptyset'$ after stage $s(k)$. Therefore, $\emptyset'\leq_\Tur D$, so $D$ computes a DNC$_2$ function (which is certainly $h$-bounded).
\end{proof}
\andre{changed notation to $J$ as elsewhere in the paper}\noam{ Good, thanks}

%%%%%%%%
%%%%%%%%
\section{The (strong) continuous covering property}
\label{sec:continuous}
%%%%%%%%
%%%%%%%%

%
% JOE (6/11/19): Still have to read the main proof
%

Recall from the introduction  that by definition, an oracle $D$ has the continuous covering property if   for every $\Pi^0_1$ class $P$ of positive measure, there is a $D$-computable tree $T\subseteq 2^{<\w}$ with no dead ends such that $\leb([T])>0$ and $[T]\subseteq P$. The main result of this section, and arguably of the paper, is that the continuous covering property does not imply PA-completeness. We also relate the continuous covering property to the discrete covering property and to $\High(\CR,\MLR)$. Everything we prove about the continuous covering property actually holds for an apparently stronger notion; see Definition~\ref{defn:strong} below.

\begin{prop}\label{prop:continuous-to-discete}
Every oracle that has the continuous covering property also has the discrete covering property.
\end{prop}
\begin{proof}
Let $\seq{C_{n,k}}$ be a computable array of independent clopen subsets of $2^\omega$ (each given canonically) such that $\leb(C_{n,k}) = 2^{-n}$ for all $k$. For an sequence $\bar A = \seq{A_n}$ of subsets of $\omega$ such that $\wt(\bar A) < \infty$. Consider the $\Sigma^0_1$ class
\[
U =  \bigcup\; \set{C_{n,k}}{k\in A_n}.
\]
Note that
\[
\leb(U) = 1 - \prod_{n\in\omega}\prod_{k\in A_n} (1 - 2^{-n}).
\]
But $\prod_{n\in\omega}\prod_{k\in A_n} (1 - 2^{-n}) > 0$ if and only if $\wt(\bar A)=\sum_{n\in\omega}\sum_{k\in A_n} 2^{-n} < \infty$. Therefore, $\leb(U) < 1$.

If $D$ has the continuous covering property, then there is an open set $V\supseteq U$ such that $\leb(V)<1$ and $S_V = \set{\sigma\in 2^{<\omega}}{[\sigma]\subseteq V}$ is $D$-computable. Let $B_n = \set{k}{C_{n,k}\subseteq V}$. Then $\bar B = \seq{B_n}$ is a $D$-computable cover of $\bar A$. All that remains is to prove that $\bar B$ has finite weight. But
\[
1 - \prod_{n\in\omega}\prod_{k\in B_n} (1 - 2^{-n}) \leq \leb(V) < 1,
\]
so $\prod_{n\in\omega}\prod_{k\in B_n} (1 - 2^{-n}) > 0$. This implies that $\wt(\bar B) = \sum_{n\in\omega}\sum_{k\in B_n} 2^{-n} < \infty$.
\end{proof}

%:
 Note that nothing prevents $[T]$  in the definition of the continuous covering property from having intervals in which it is nonempty but has measure zero (or even from having isolated paths). It is convenient to work with an apparently stronger notion in which such intervals are explicitly forbidden.

\begin{defn}\label{defn:strong}
We say that $D$ has the \emph{strong continuous covering property} if for every $\Pi^0_1$ class $P$ of positive measure, there is a $D$-computable tree $T\subseteq 2^{<\w}$ such that % $\leb([T])>0$, 
$T\ne\emptyset$, $[T]\subseteq P$, and for all $\s\in T$, $\leb([T]\cap [\s])>0$.
\end{defn}

One reason that the \emph{strong} continuous covering property is convenient is that we can show that there is a ``universal'' $\Pi^0_1$ class for this property (compare with Remark~\ref{rmk:universal_instance_of_DCP}). Let $U$ be the first component of the standard universal Martin-L\"of test, i.e., the test obtained by combining all Martin-L\"of tests. So $\leb(U)<1$ and for any Martin-L\"of test $\seq{V_n}$ there is an $n$ such that $V_n\subseteq U$. Let $\PP$ be the complement of $U$, so it is a positive measure $\Pi^0_1$ class. The following lemma states that having the strong continuous covering property for $\PP$ is enough to ensure the strong continuous covering property in general.

For any $W\subseteq 2^\omega$ and $\sigma\in 2^{<\omega}$, let $W|\s = \set{X\in 2^\w}{\s\conc X\in W}$.

\begin{lem}\label{lem:universal}
Suppose $T$ is a nonempty tree such that
 % $\leb([T])>0$, 
 $[T]\subseteq \PP$, and for all $\s\in T$, $\leb([T]\cap [\s])>0$. Then  $T$ has the strong continuous covering property. 
\end{lem}
\begin{proof}
We work in the dual setting, with $\Sigma^0_1$ classes. Let $T$ be a tree as described in the statement of the lemma. Let $W$ be the %$\Sigma^0_1$ 
open class generated by $2^{<\omega}\smallsetminus T$. So $W$ is $T$-computable, $\leb(W)<1$, and $U\subseteq W$, where $U = 2^{\omega}\smallsetminus\PP$ is the first component of the standard universal Martin-L\"of test, as above. Also, we have
\[
2^{<\omega}\smallsetminus T = \set{\s\in 2^{<\w}}{[\s]\subseteq W} = \set{\s\in 2^{<\w}}{\leb(W|\s)=1}.
\]

Let $V$ be a $\Sigma^0_1$ class with $\leb(V)<1$. Let $S\subseteq 2^{<\omega}$ be a prefix-free c.e.\ set of strings such that $V=[S]$. Define $S^n$ recursively, as usual: let $S^0 = \{\seq{}\}$ and define $S^{n+1}$ to be $\set{\s\conc \tau}{\s\in S^n\text{ and }\tau\in S}$. It is straightforward to check that $\leb([S^n]) = \left(\leb(V)\right)^n$, so an effective subsequence of $\seq{[S^n]}$ forms a Martin-L\"of test. Therefore, there is an $n$ such that $[S^n]\subseteq W$. Let $n$ be the least such; $n>0$ since $\leb(W)<1$. Let $\s$ be a string witnessing that $[S^{n-1}]\nsubseteq W$, i.e., $\s\in S^{n-1}$ and $[\s]\nsubseteq W$. Now consider $W|\s$. We have that $[\s]\nsubseteq W$ implies that $\leb(W|\s)<1$, and $[S^n]\subseteq W$ implies that $V = [S]\subseteq W|\s$. Note that
\[
\{\tau\,:\, [\tau]\subseteq W|\s\} = \{ \tau\,:\, [\s\conc\tau]\subseteq W\}
\]
is $T$-computable. Finally, if $\leb((W|\s)|\tau) = \leb(W|\s\conc\tau)=1$, then $[\s\conc\tau]\subseteq W$, hence $[\tau]\subseteq W|\sigma$. Since $V$ was an arbitrary $\Sigma^0_1$ class with $\leb(V)<1$, we have proved that $\deg_\Tur(T)$ has the strong continuous covering property.
\end{proof}

\begin{prop}\label{prop:high-continuous}
Every oracle $D$  in $\High(\CR,\MLR)$ has the strong continuous covering property. 
\end{prop}
\begin{proof}
By Proposition~\ref{prop:single-martingale}, there is a $D$-computable martingale $N$ that succeeds on all non-ML-random sequences. We may assume that   $N(\seq{})<1$. Let $Q$ be the set of minimal strings $\sigma$ with $N(\s)\geq 1$ and let $V = [Q]$. \andre{really the same notation for opposites? I prefer $[Q]^\preceq$ } \noam{ Should be clear from the context. We used it above.} Note that if $(\forall \tau\preceq\sigma)\; N(\tau)<1$, then $[\sigma]\nsubseteq V$, and in fact $\leb(V|\sigma)\leq N(\sigma)<1$. Thus
\[
\set{\sigma\in 2^{<\omega}}{[\sigma]\subseteq V} = \set{\sigma\in 2^{<\omega}}{(\exists\tau\preceq\sigma)\; N(\tau)\geq 1},
\]
which is $D$-computable.

Let $U = 2^{<\omega}\smallsetminus\PP$ be the first component of the standard universal Martin-L\"of test and fix a prefix-free set $S\subseteq 2^{<\omega}$ such that $U=[S]$. We attempt to build a sequence of strings $\sigma_0\prec\sigma_1\prec\sigma_2\prec\cdots$ as follows. Let $\sigma_0 = \seq{}$. If $\sigma_n$ has been defined, it must be the case that $[\sigma_n]\nsubseteq V$. If possible, pick a $\tau\in S$ such that $[\sigma_n\conc\tau]\nsubseteq V$ and let $\sigma_{n+1} = \sigma_n\conc\tau$.

If $\sigma_n$ exists for every $n$, then let $X = \bigcup_{n\in\omega}\sigma_n$. Note that $X\notin V$, so $N$ does not succeed on $X$; in fact, it never reaches $1$. On the other hand, $X\in \bigcap_{n\in\omega} [S^n]$, so it is not Martin-L\"of random. This contradicts the choice of $N$.

Therefore, there is an $n$ such that $\sigma_n$ is defined, but $\sigma_{n+1}$ is not. So $[\sigma_n]\nsubseteq V$, but for every $\tau\in S$, we have $[\sigma_n\conc\tau]\subseteq V$. This means that $U\subseteq V|\sigma_n$. We also have that $\leb(V|\sigma_n)<1$. Finally, if $[\rho]\nsubseteq V|\sigma_n$, then it must be the case that $\leb((V|\sigma_n)|\rho) = \leb(V|\sigma_n\conc\rho) \leq N(\sigma_n\conc\rho) < 1$. Therefore, Lemma~\ref{lem:universal} tells us that $D$ has the strong continuous covering property. 
\end{proof}

\begin{thm}\label{thm:main-separation}
There is an oracle $D$ with the strong continuous covering property that does not have PA degree.
\end{thm}
\begin{proof}
The proof of this theorem is an elaboration on the proof of Theorem~\ref{thm:K-bounded}. We will build a tree~$T$ satisfying the hypothesis of Lemma~\ref{lem:universal}, which does not have PA degree. The tree~$T$ is built by forcing. In the previous proof, forcing conditions specified a finite initial segment~$\s$ of the $K$-compression function we built, a $\Pi^0_1$ class of possible extensions of $\sigma$, and a rational number~$q$ with the promise that the function that we eventually build will have weight at most~$q$. In the current construction, a forcing condition will specify: a finite initial segment of~$T$ (which we code by its set of leaves~$u$); a $\Pi^0_1$ class of possible extensions of~$u$ to trees $S\subseteq \PP$; and for each leaf $\s\in u$, a rational number~$q_\s$ with the promise that the measure of $T\cap [\s]$ is at least~$q_\s$. The structure of the proof is the same as before, but the combinatorial lemmas are more elaborate. We start with some terminology and notation. 

\smallskip

By a tree (\emph{arbre} in French) we mean a subset of $2^{<\w}$ closed under taking initial segments. Note that there is a 1-1 correspondence between closed subsets of $2^\w$ and subtrees of $2^{<\w}$ with no dead ends.   Let $\TT$ denote  the set of nonempty trees with no dead ends. Coding strings by numbers, $\TT$ itself is an effectively closed subset of Cantor space.
We will work with $\Pi^0_1$ classes of trees with no dead ends, namely, $\Pi^0_1$ subclasses of $\TT$
. 
To keep notational complexity in check, below, we ignore the difference between $[T]$ and $T$ (for $T\in \TT$) and write $T$ for both. Note that for $S,T\in \TT$, we have $S\subseteq T$ iff $[S]\subseteq [T]$. The operation of intersection is well defined; for $S,T\in \TT$, we let $S\cap T$ be the unique element $R$ of $\TT$ such that $[R] = [S]\cap [T]$, unless $[S]\cap[T]$ is empty, a case which we will avoid.

Infinite trees are built up of finite ones. Let $\TT_{<\w}$ be the collection of all nonempty finite subtrees of $2^{<\w}$. For $T\in \TT$ and $\vt\in \TT_{<\w}$, we say that $T$ \emph{extends} $\vt$ (and sometimes write $\vt\prec T$) %\hl{do we?} %%% we do now) 
if $\vt\subset T$ and every $\s\in T$ is comparable with a leaf of $\vt$. For each $\vt\in \TT_{<\w}$, we let $[\vt]$ be the collection of $T\in \TT$ that extend $\vt$. This is a clopen subset of $\TT$, and the collection of these sets generates the topology on $\TT$. We often restrict ourselves to trees of a fixed height; for $n<\w$, let $\TT_n$ be the set of finite trees all of whose leaves have length $n$. For $\vt\in \TT_{<\w}$ and $n$ greater than the height of $\vt$, we let $[\vt]_n$ be the set of trees $\vp\in \TT_n$  which extend $\vt$, again in the sense that each $\tau\in \vp$ extends some $\s\in \vt$ and each $\s\in \vt$ is extended by some $\tau\in \vp$; we write $\vt\preceq \vp$. Note that for $\vt,\vp\in \TT_{<\w}$, $\vt\preceq \vp$ if and only if $[\vp]\subseteq [\vt]$. 
% \andre{Does this mean compatible?} \noam{ Not sure what you mean by compatibility here; do we use this with respect to finite trees?} \andre{comment withdrawn}

We also implicitly use the bijection between $\TT_{<\w}$ and the collection of finite antichains of strings (a tree is mapped to its leaves). For example, for a finite antichain of strings $u$ we let $[u]$ be $[\vt]$ where $u$ is the set of leaves of $\vt$. A tree $\vt\in \TT_{<\w}$ and its set of leaves are both identified with the clopen subset of $2^{\w}$ determined by $\vt$. Thus for example, for a finite antichain $u$ of strings we let $\leb(u) = \sum_{\s\in u}2^{-|\s|}$. Similarly, for $T\in \TT$ and $\tau\in 2^{<\w}$ we let $T\cap \tau = \{\s\in T\,:\, \s\not\perp\tau\}$. 

\bigskip
Fix the $\Pi^0_1$ class $\PP$ from Lemma~\ref{lem:universal}. Our forcing conditions are triples $(u,P,\bar q)$ such that:
\begin{itemize}
	\item $u$ is a nonempty finite antichain of strings;
	\item $P\subseteq \TT$ is a $\Pi^0_1$ subclass of $[u]$ such that:
	\begin{itemize}
		\item for all $T\in P$ we have $T\subseteq \PP$; 
		% \item every $T\in P$ extends $\vt$, i.e.\ $T\cap 2^n = \vt$;
		\item if $T\in P$, $S\in [u]$ and $S\subseteq T$ then $S\in P$. 
	\end{itemize}
	\item $\bar q  = \seq{q_\s}_{\s\in u}$ is a sequence of positive rational numbers smaller than $2^{-|\s|}$, and 
	\[ P_{\ge \bar q} = \{ T\in P\,:\, (\forall\s\in u)\; \leb(T\cap\s) \ge q_\s \} \] is nonempty.
\end{itemize}
\andre{some intuition how this works would be helpful. We seem to assume it is enough to have gotten through the warmup}
\noam{ The warmup really does contain all the main ideas. I'm not sure I have much more intuition to provide...}
 \andre{something on how we get from strings to trees, that seems to be the main combinatorial addition}
 \noam{Paragraph added to the beginning of the proof.} \andre{OK helps}

If we let $P$ be the set of trees $T\in \TT$ such that $T\subseteq \PP$ and $q$ be any rational number smaller than $\leb(\PP)$, then $(\{\seq{}\},P,\seq{q})$ is a condition. So the set of conditions is nonempty. A condition $(v,R,\bar r)$ \emph{extends} a condition $(u,P,\bar q)$ if:
\begin{enumerate}
	\item $u\preceq v$;
	\item $R\subseteq P$; and
	\item for all $\s\in u$, we have $q_\s \le \sum \set{r_\tau}{\tau\in v \andd \tau\succeq \s}$.
\end{enumerate}
Note that if $u\preceq v$, then condition (3) is equivalent to 
$[v]_{\ge \bar r}\subseteq [u]_{\ge \bar q}$. In particular, we see that if a condition $(v,R,\bar r)$ extends a condition $(u,P,\bar q)$, then $R_{\ge \bar r} \subseteq P_{\ge \bar q}$. 

\medskip
Our first lemma is directly analogous to Lemma~\ref{lem:lose_weight_now_ask_me_how}.

\begin{lem}\label{lem:gain_weight_now_heres_a_cake}
Let $(u,P,\bar q)$ be a condition. Then 
\[
P_{>\bar q} =  \{ T\in P\,:\, (\forall \s\in u)\; \leb(T\cap \s) > q_\s \}
\]
is nonempty.
\end{lem}

\begin{proof}
Suppose not.  
Let $v$ be a $\subseteq$-maximal subset of $u$ for which there is some $T\in P_{\ge \bar q}$ with $\leb(T\cap \s)>q_\s$ for all $\s\in v$, and let~$T$ witness this. Let $\epsilon>0$ be rational smaller than $\leb(T\cap \s)-q_\s$ for all $\s\in v$, and let $q'_\s = q_\s+\epsilon$ for $\s\in v$ and $q'_\s = q_\s$ for $\s\in v- u$. Thus, $P_{\ge \bar q'}$ is nonempty, and for all $S\in P_{\ge \bar q'}$, for all $\s\in v-u$ we have $\leb(S\cap \s) = q_\s$. Choose any $\s\in v-u$, and let $Q = \{ S\cap \s\,:\, S\in P_{\ge \bar q'}\}$. Let $q = q_\s$. So $Q$ is a nonempty $\Pi^0_1$ subclass of $\TT$ and for all $T\in Q$, $T\subseteq \PP$ and $\leb(T) = q$.
 	
Let $V_n = \emptyset$ if $n\notin \emptyset'$, and otherwise let $V_n = \{ \s\conc 0^n\,:\, |\s|=s\}$ where $s$ is the stage at which $n$ enters $\emptyset'$. Since $\leb(V_n)\le 2^{-n}$ and $\seq{V_n}$ is uniformly c.e., for all sufficiently large $n$ we have $V_n \cap \PP = \emptyset$. 

Let $m<\w$. By compactness, we can effectively find some $t<\w$ and some $C\subseteq \TT_t$ such that $Q \subseteq \bigcup_{\vt\in C} [\vt]$ and such that $q/\leb(\vt) > 1-2^{-m}$ for all $\vt\in C$. We then claim that provided that $m$ is large enough, $m\in \emptyset'$ if and only if $m\in \emptyset'_t$. For fix some $\vt\in C$ such that $[\vt]\cap Q\ne \emptyset$, and fix some $T\in [\vt]\cap Q$. If $m$ enters $\emptyset'$ at stage $s>t$ then for every leaf $\s$ of $\vt$, $\leb(T|\s) \le \leb(\PP|\s) \le 1-2^{-m}$ and so $q=\leb(T) \le (1-2^{-m})\leb(\vt)$ which is not the case. This algorithm for computing $\emptyset'$ gives the desired contradiction.
\end{proof}

We will often use Lemma~\ref{lem:gain_weight_now_heres_a_cake} in conjunction with the following:

\begin{lem} \label{lem:extending_the_finite_part}
	Let $(u,P,\bar q)$ be a condition; let $v\succeq u$, and suppose that $S\succ v$ and $S\in P_{>\bar q}$. Then there is some $\bar p = \seq{p_\tau}_{\tau\in v}$ such that $(v,P\cap [v], \bar p)$ is a condition extending $(u,P,\bar q)$. 
\end{lem}

\begin{proof}
	For $\s\in u$, let $v_\s = \set{\tau\in v}{\tau\succeq \s}$. Choose rational $p_\tau$ for $\tau\in v$ so that $p_\tau \le \leb(S\cap \tau)$, and for all $\s\in u$, $\sum_{\tau\in v_\s} p_\tau \ge q_\s$; this is possible because $S\cap \s = \bigcup_{\tau\in v_\s} {S\cap \tau}$ and so $\sum_{\tau\in v_\s} \leb(S\cap \tau) = \leb(S\cap \s)> q_\s$. Then $S\in [v]\cap P_{\ge \bar p}$ so $(v,P\cap [v], \bar p)$ is indeed a condition as required. 
\end{proof}

For a filter $G$ of forcing conditions, we let $T_G$ be the downward closure of
\[ \bigcup u \Cyl{ (u,P,\bar q)\in G \text{ for some $P$ and $\bar q$}}.\]
We assume from now that $G$ is fairly generic.

\begin{lem}\label{lem:generic_tree_is_infinite}
	$T_G\in \TT$.
\end{lem}

\begin{proof}
	It suffices to show that for any condition $(u,P,\bar q)$, for all large~$n$, there is an extension $(v,Q,\bar p)$ of $(u,P,\bar q)$ such that every $\s\in v$ has length~$n$. 

	Let $(u,P,\bar q)$ be a condition, and let~$n> |\s|$ for all $\s\in u$. By Lemma~\ref{lem:gain_weight_now_heres_a_cake}, let $S\in P_{>\bar q}$. Let $v = S^{=n}$ be the collection of strings on~$S$ of length~$n$. Since $u\prec S$, we have $u\prec v$, and of course $v\prec S$; by Lemma~\ref{lem:extending_the_finite_part}, there is some~$\bar p$ such that $(v,P\cap [v],\bar p)$ is a condition extending $(u,P,\bar q)$. 
	%
	% let $T^*\in P_{> \bar q}$. Let $I = \{ i<2\,:\, \s\conc i \in T^*\}$ which is nonempty. Let
	% \[ u' = \left( u-\{\s\} \right) \cup \{ \s\conc i\,:\, i\in I\}.\]
	% Define $\bar q'$ by extending $\bar q$ but replacing $q_\s$ by $q_{\s\conc i}$ for $i\in I$, so that $q_{\s\conc i}\le \leb(T^*\cap(\s\conc i))$ and $\sum_{i\in I} q_{\s\conc i}\ge q_\s$. Then $(u', P\cap [u'], \bar q')$ is a condition extending $(u,P,\bar q)$ and $\s$ has a proper extension in $u'$. 
 \end{proof}

\begin{lem}\label{lem:stupid_stuff}
	Let $(u,P,\bar q)\in G$. Then $T_G\in P_{\ge \bar q}$. 
\end{lem}

\begin{proof}
	Let $\vt\prec T_G$; we can find some $(v,Q,\bar p)\in G$ extending $(u,P,\bar q)$ such that $\vt \preceq v$. Since $Q_{\ge\bar p}\subseteq  [v]\cap  P_{\ge \bar q}$, it follows that $[\vt]\cap P_{\ge \bar q}$ is nonempty. Since $P_{\ge\bar q}$ is closed, the lemma follows. 
\end{proof}

Let $\Gamma\colon \TT\to 2^\w$  be a Turing functional. Let $D_\Gamma$ be the set of conditions $(u,P,\bar q)$ such that $\Gamma(T)\notin \DNC_2$ for all $T\in P_{\ge \bar q}$. We show that $D_\Gamma$ is dense. 

First we prepare. The following is analogous to Lemma~\ref{lem:compression_functions:breathing_room}. We define the collection~$F$ of conditions that give us sufficient breathing room. Let $(v,P,\bar q)$ be a condition; for $\s\in v$, let $\epsilon_\s = (2^{-|\s|}-q_\s)/3$. We set $r_\s = 2^{-|\s|}$, so that we can write $\bar q = \bar r-3\bar \epsilon$. The condition $(v,P,\bar q)$ is in~$F$ if $P_{>\bar r-\bar \epsilon}$ is nonempty. 

\begin{lem} \label{lem:tree_preparation}
	The collection~$F$ of conditions is dense. 
\end{lem}
\andre{already for the K-version I thought why  not we can make F the forcing p.o.. What would happen to the lemma then?} \noam{ If you do not add the definition of a ``pre-condition'' (what we currently call a condition, but not necessarily in F), then you would need to add the proof of this lemma each time we produce conditions, e.g. to Lemma 5.7, Lemma 5.8, and the end of the proof.} \andre{ah ok thx}

\begin{proof}
	Let $(u,P,\bar q)$ be a condition. 
	By Lemma~\ref{lem:gain_weight_now_heres_a_cake}, let $T\in P_{>\bar q}$. Fix some $\s\in u$. Take a positive rational number $\delta_\s$ such that $6\delta_\s < \leb(T\cap \s) - q_\s$. Find some finite antichain $v_\s$ of extensions of $\s$ such that $v_\s\prec T\cap \s$ and further $\leb(v_\s)-\leb(T\cap \s)< \delta_\s$ (where we again identify $v_\s$ with the clopen subset of Cantor space it determines). 

	For $\tau\in v_\s$, let $\eta_\tau = r_\tau - \leb(T\cap \tau)$; so
	\begin{equation} \label{eqn:basic_eta}
			\sum_{\tau\in v_\s} \eta_\tau = \leb(v_\s) - \leb(T\cap v_\s) <\delta_\s,
	\end{equation}
	using the fact that $T\cap \s = T\cap v_\s$. 

	Let $u^*_\s = \left\{ \tau\in v_\s \,:\,  r_\tau - 3\eta_\tau > 0\right\}$. We aim to show that:
	\begin{equation}
		\label{eqn:goal}
		\sum_{\tau\in u^*_\s} (r_\tau-3\eta_\tau) > q_\s.
	\end{equation}
	If this is the case, then each $u^*_\s$ is nonempty; letting $u^* = \bigcup_{\s\in u} u^*_\s$, we would have $u\preceq u^*$. We can then choose, for each $\tau \in u^*$, a rational $\epsilon_\tau$ just slightly larger than~$\eta_\tau$, so that we still have $\epsilon_\tau < r_\tau/3$ and $\sum_{\tau\in u^*_\s}(r_\tau - 3\epsilon_\tau) > q_\s$ for each $\s\in u$. Then $(u^*,P\cap [u^*], \bar r-3\bar \epsilon)$ would be a condition extending $(u,P,\bar q)$; it would be a condition in~$F$, since $T^* = T\cap u^*$ witnesses that $(P\cap [u^*])_{>\bar r-\bar\epsilon}$ is nonempty: $T^*\subseteq T$ and so is in~$P$, and for $\tau\in u^*$ we have $\leb(T^*\cap \tau) = \leb(T\cap \tau) = r_\tau-\eta_\tau > r_\tau -\epsilon_\tau$.

	\smallskip
	
	Fix $\s\in u$. Toward showing \eqref{eqn:goal}, we note that by \eqref{eqn:basic_eta}, as $\sum_{\tau\in u^*_\s}\eta_\tau \le \sum_{\tau\in v_\s}\eta_\tau$, we have 
	\[
		\sum_{\tau\in u^*_\s} (r_\tau - 3\eta_\tau) > \leb(u^*_\s) - 3\delta_\s, 
	\]
	so it suffices to show that 
	\begin{equation}
		\label{eqn:the_measure_of_u_star}
		\leb(u^*_\s) \ge q_\s + 3\delta_\s. 
	\end{equation}

	Let $w_\s = v_\s \setminus u^*_\s = \left\{ \tau\in v_\s \,:\,  \leb(T | \tau) \le 2/3 \right\}$. Then
	\[
		\leb(u^*_\s)  + \leb(w_\s) = \leb(v_\s) \ge \leb(T\cap \s) > q_\s + 6\delta_\s; 
	\]
	so it suffices to show that 
	\begin{equation}
		\label{eqn:bound_on_w_sigma}
		\leb(w_\s) \le 3\delta_\s. 
	\end{equation}

	By definition, $\leb(T\cap w_\s) \le (2/3)\cdot \leb(w_\s)$. Now
	\begin{multline*}
		 \leb(u^*_\s) + (2/3)\cdot \leb(w_\s) + \delta_\s \ge \leb(T\cap u^*_\s) + \leb(T\cap w_\s)+ \delta_\s = \\ \leb(T\cap \s) + \delta_\s \ge \leb(v_\s) = \leb(u^*_\s) + \leb(w_\s); 
	\end{multline*}
	subtracting $\leb(u^*_\s)$ gives the desired result~\eqref{eqn:bound_on_w_sigma}. 
\end{proof}

Now fixing a condition $(u^*,P^*, \bar r - 3\bar \epsilon)$ in~$F$, we find an extension in~$D_\Gamma$. The main property we use is that for $S,T\in [u^*]_{>\bar r-\bar \epsilon}$, for all $\tau\in u^*$ we have $\leb(S\cap T\cap \tau)> r_\tau - 2\epsilon_\tau$, so $S\cap T \in [u^*]_{>\bar r - 2\bar \epsilon}$. We can now run the proof from above. 

% Of course what will be used is that $T\in \TT$ 
% is in $[\vp]_{\ge \bar s}$ if and only if for all sufficiently large $n$ we have $T\rest{2^{\le n}}\in 
% [\vp]_{\ge \bar s}$. 

\smallskip

We define a partial computable process which may output 0 or 1; by the recursion theorem, we obtain some~$e$ such that this output is $J(e)$. Let 
\[
	C = \left\{ \vs\in \TT_{<\w} \,:\,  \vs\in [u^*]_{>\bar r-2\bar \epsilon} \andd \Gamma(\vs,e)\converge \right\}, 
\]
and let
\[
	Q = \left\{ T\in P^* \,:\,  (\forall \vt \subset T)\,\,\vt\notin C \right\}.
\]
% We let~$Q$ be the set of $T\in P^*$ which are not removed by finding some $\vt\in \TT_{<\w}$ with $\vt\subset T$, $\vt\succeq u^*$, $\leb(\vt\cap \tau) > r_\s-2\epsilon_\tau$ for all $\tau\in u^*$, and such that $\Gamma(\vt,e)\converge$. Note that 
Here by $\vt\subset T$ we do mean the sets of strings, not the associated closed sets; and we do not require that $T\succ \vt$. 

Note that for all $T\in P^*$, since $T\subseteq \PP$, for all $\tau\in T$ we must have $\leb(T\cap \tau) < 2^{-|\tau|}$. Hence if $T\in [u^*]_{\ge \bar p}$ for some $\bar p$, then for all $\vt\succeq u^*$ with $\vt\prec T$, we must have $\vt\in [u^*]_{>\bar p}$. Hence, if $Q_{\ge \bar r-2\bar \epsilon}$ is nonempty, then $(u^*,Q,\bar r - 2\bar \epsilon)$ is an extension of $(u^*,P^*,\bar r-3\bar \epsilon)$ in $D_\Gamma$. We suppose then that~$Q_{\ge \bar r - 2\bar \epsilon}$ is empty.

As in Lemma~\ref{lem:compression_functions:the_set_C}, by compactness, we can find some $n<\w$ and a set $E\subseteq \TT_n$ such that:
 \begin{enumerate}
 	\item $E\subset [u^*]_{>\bar r -2\bar \epsilon}$;
 	\item For every $\vt\in E$ there is some $\vr\subseteq \vt$ in~$C$;
 	\item For every $\vt\in E$ and $\vr\subseteq \vt$ in $[u^*]_{>\bar r -2\bar \epsilon}$ we have $\vr\in E$;
 	\item There is some $\vt\in E$ such that $[\vt]\cap P^*_{>\bar r-\bar \epsilon}$ is nonempty.
 \end{enumerate}
The proof is the same; we let~$E$ be the set of $\vs\in \TT_n\cap [u^*]_{>\bar r-2\bar \epsilon}$ on a computable tree determining~$P^*$, for some~$n$ such that every $\vt\in \TT_n$ on a tree determining~$Q$ is in $[u^*]_{\le \bar r-2\bar \epsilon}$. 

We let $\widehat E = E \cap [u^*]_{>\bar r-\bar \epsilon}$. As observed above, if $\vt,\vr\in \widehat E$ then $\vt\cap \vr\in [u^*]_{>\bar r -\bar \epsilon}$ and so $\vt\cap\vr\in E$. As in the proof of Lemma~\ref{lem:getting_one_colour}, this shows that there is some $i\in \{0,1\}$ such that for every $\vt\in \widehat E$ there is some $\vs\subseteq \vt$ in $C_i = \left\{ \vs\in C \,:\,  \Gamma(\vs,e)=i \right\} $. As above, this~$i$ is the output of our computable process, so $J(e)=i$. 

Let $\vt\in \widehat E$ such that $[\vt]\cap P^*_{>\bar r - \bar \epsilon}$ is nonempty; fix some $T$ in that set. Find some $\vs\subseteq \vt$ in $C_i$. Let $S = T\cap \vs$. Note that $\vs\prec S$ because $\vs\subseteq \vt$. For all $\s\in u^*$, as $\leb(T\cap \s)> r_\s-\epsilon_\s$ and $\leb(\vs\cap \s) > r_\s - 2\epsilon_\s$, we have $\leb(S\cap \s)> r_\s-3\epsilon_\s$. In particular,~$S$ is infinite; as $S\subseteq T$, we have $S\in P^*$. Altogether, $S\in P^*_{>\bar r -3\bar \epsilon}$. By Lemma~\ref{lem:extending_the_finite_part}, as $\vs\succeq u^*$ and $S\succeq \vs$, there is some condition $(\vs,P^*\cap [\vs],\bar p)$ extending $(u^*,P^*,\bar r-3\bar \epsilon)$; this condition is in~$D_\Gamma$. 
% \hl{Blah blah blah do the same.} For the final step, suppose that $\Gamma(\vt,e)\converge = i$, with $\vt\in \TT_{<\w}$, $\vt\in[u^*]_{>\bar r -2\bar \epsilon}$ and $S\subset T$ for some $T\in P^*_{>\bar r -\bar \epsilon}$. Then $T\cap S\in P^*$ and $T\cap S\in [u^*]_{>\bar r - 3\bar \epsilon}$. 
\end{proof}

%%%%%%%%
%%%%%%%%
\section{(Weak) strong weak weak K{\H o}nig's lemma}
%%%%%%%%
%%%%%%%%

%
% JOE (6/11/19): Proofs to write
%

In this section, we study two (possibly equivalent) reverse mathematical principles strictly between \sfup{WKL} (weak K{\H o}nig's lemma) and \sfup{WWKL} (weak weak K{\H o}nig's lemma). Our principles correspond to the continuous covering property and its strong variant. We assume that the reader has some familiarity with reverse mathematics; see Simpson~\cite{S:09} for an introduction.

We say that a tree $T\subseteq 2^{<\omega}$ has \emph{positive measure} if there is a $\epsilon>0$ such that
\[
(\forall n)\; \frac{\#\set{\sigma\in T}{|\sigma|=n}}{2^n} > \epsilon.
\]
In the introduction, we defined \emph{strong weak weak K{\H o}nig's lemma} ({\sffamily\itshape SWWKL}):

\medskip
\hangindent=30pt
\noindent\hspace{30pt}%
If $T\subseteq 2^{<\omega}$ is a tree with positive measure, then there is a nonempty subtree $S\subseteq T$ such that if $\sigma\in S$, then $S$ has positive measure above $\sigma$.

\medskip
\noindent
In particular, note that \sfup{SWWKL} implies that $S$ is a \emph{perfect} subtree of $T$. Our second principle is the one corresponding to the continuous covering property: \emph{weak strong weak weak K{\H o}nig's lemma} ({\sffamily\itshape WSWWKL}):

\medskip
\hangindent=30pt
\noindent\hspace{30pt}%
If $T\subseteq 2^{<\omega}$ is a tree with positive measure, then there is a subtree $S\subseteq T$ of positive measure that has no dead ends.

\medskip
\noindent
We will prove, over $\sfup{RCA}_0$, that
\begin{center}
\vspace{4pt}
\begin{tikzpicture}[text centered, node distance=0.75cm, auto]
	\node (wkl) {\sfup{WKL}};
	\node[right=of wkl] (swwkl) {\sfup{SWWKL}};
	\node[right=of swwkl] (wswwkl) {\sfup{WSWWKL}};
	\node[right=of wswwkl] (wwkl) {\sfup{WWKL.}};

	\path ([yshift=2pt]wkl.east) edge[imp] ([yshift=2pt]swwkl.west)
		(swwkl) edge[imp] (wswwkl)
		([yshift=2pt]wswwkl.east) edge[imp] ([yshift=2pt]wwkl.west);
		
	\path ([yshift=-2pt]wwkl.west) edge[imp,->] node[yshift=6pt] {\tiny\bf/} ([yshift=-2pt]wswwkl.east)
		([yshift=-2pt]swwkl.west) edge[imp,->] node[yshift=6pt] {\tiny\bf/} ([yshift=-2pt]wkl.east);
\end{tikzpicture}
\end{center}
It is easy to see that \sfup{SWWKL} implies \sfup{WSWWKL}; we do not know whether the reverse implication holds. The remaining implications and non-implications are proved below in Propositions~\ref{prop:RM1}--\ref{prop:RM4}.

\begin{prop}\label{prop:RM1}
$\sfup{RCA}_0 + \sfup{WKL} \vdash \sfup{SWWKL}$.
\end{prop}
\begin{proof}
This is the formalisation in $\sfup{RCA}_0$ of the fact that every PA degree has the strong continuous covering property. 
\noam{Do we need any more?}
\end{proof}

\begin{prop}
$\sfup{RCA}_0 + \sfup{WSWWKL} \vdash \sfup{WWKL}$.
\end{prop}
\begin{proof}
\sfup{WWKL} simply says that if $T$ is a tree with positive measure, then $T$ has an infinite path. Given a tree $T$ with positive measure, let $S\subseteq T$ be the positive measure subtree with no dead ends that is guaranteed by \sfup{WSWWKL}. Since $S$ has positive measure, it must contain the root. Since it has no dead ends, we can construct an infinite path $X$ by always following the leftmost branch in $S$. Then $X$ is an infinite path in $T$.
\end{proof}

\begin{prop}
$\sfup{RCA}_0 + \sfup{WWKL} \not\vdash \sfup{WSWWKL}$.
\end{prop}
\begin{proof}
Fix an $\omega$-model $(\omega,\+S)$ of \sfup{WWKL} such that whenever $X\in\+S$, there is an incomplete Martin-L\"of random sequence $Y\in\+S$ such that $X\leq_\Tur Y$. Building such a model is straightforward; for example, we can let~$\+S$ be the ideal generated by the joins of finitely many columns of some incomplete ML-random sequence. We claim that $(\omega,\+S)$ is not a model of \sfup{WSWWKL}.

Assume that $(\omega,\+S)$ actually is a model of \sfup{WSWWKL}. By formalizing Propositions~\ref{prop:continuous-to-discete} and~\ref{prop:sdnc}, it must be the case that for any order function $h\colon\omega\to\omega\smallsetminus\{0,1\}$, there is an $h$-bounded DNC function in $\+S$. However, as mentioned in the introduction, Bienvenu and Porter~\cite{BP:16} showed that for a sufficiently slow growing $h$, only complete Martin-L\"of random sequences can compute $h$-bounded DNC functions. This is a contradiction, so $(\omega,\+S)$ is not a model of \sfup{WSWWKL}.
\end{proof}

In order to construct a model of \sfup{SWWKL} that is not a model of \sfup{WKL}, we need to strengthen Theorem~\ref{thm:main-separation}.

\begin{thm}
Assume that $X$ does not have PA degree. There is an oracle $D$ with the strong continuous covering property relative to $X$ such that $X\oplus D$ does not have PA degree.
\end{thm}
\begin{proof}
We modify the proof of Theorem~\ref{thm:main-separation}. We use the same forcing notion, except that we replace $\PP$ by~$\PP^X$. Rather than using the recursion theorem, we argue as follows. Let~$\Gamma$ be an $X$-computable functional, and let $(u^*,P^*,\bar r-3\bar \epsilon)$ be a condition in~$F$. For every~$e$, let $\psi(e)$ be the output of the computable process described in the proof above when computing $\Gamma(\vs,e)$. The function~$\psi$ is $X$-partial computable. If~$\psi$ is not total, then the corresponding extension $(u^*,Q,\bar r-2\bar \epsilon)$ is an extension forcing that $\Gamma(h,e)\diverge$. Otherwise, $\psi$ is an $X$-computable function. By assumption, $\psi\notin \DNC_2$; so there is some~$e$ such that $\psi(e)= J(e)$. The corresponding extension $(\vs,P^*\cap [\vs],\bar p)$ then forces that $\Gamma(h,e) = J(e)$. 
% \hl{[TO BE WRITTEN]}
\end{proof}

\begin{prop}\label{prop:RM4}
$\sfup{RCA}_0 + \sfup{SWWKL} \not\vdash \sfup{WKL}$.
\end{prop}
\begin{proof}
By iterating the previous result, build an $\omega$-model $(\omega,\+S)$ such that
\begin{itemize}
\item $\+S$ is a Turing ideal;
\item for every $X\in\+S$, there is a $D\in\+S$ with the strong continuous covering property relative to $X$;
\item $\+S$ contains no set of PA degree.
\end{itemize}
Therefore, $(\omega,\+S)\models\sfup{RCA}_0 + \sfup{SWWKL} + \neg\sfup{WKL}$.
\end{proof}

We should mention a connection to recent work of Chong, Li, Wang, and Yang~\cite{CLWY:}, who studied the complexity of computing perfect subsets of sets of positive measure. They report that during a discussion with Wei Wang about their work, Ludovic Patey proved that
\[
\sfup{RCA}_0 + \text{``every closed set of positive measure has a perfect subset''} \not\vdash \sfup{WKL}.
\]
Note that the principle \sfup{SWWKL} implies that every closed set of positive measure has a perfect subset \emph{of positive measure}, so Proposition~\ref{prop:RM4} improves on the result of Patey. More recently, Barmpalias and Wang have announced that they showed that $\sfup{RCA}_0$ together with the principle ``every closed set of positive measure has a perfect subset of positive measure'' does not imply $\sfup{WKL}$.

\bibliographystyle{plain}
\bibliography{references}

\end{document}